\documentclass[a4paper,11pt]{amsart}%
\usepackage{pdfsync}
\usepackage{dsfont}
\usepackage{color}
\usepackage{amsthm}
\usepackage{enumerate}
\usepackage{amsfonts}
\usepackage{amssymb}%
\usepackage{mathtools}
\usepackage{braket}
\usepackage{bm}
\usepackage{amsmath}
\usepackage{caption}
\usepackage{fullpage}
\usepackage{pb-diagram}
\usepackage{graphicx}
\usepackage{subcaption}
\usepackage{hyperref}
\usepackage{cite}
\hypersetup{colorlinks,breaklinks,
             linkcolor=blue,urlcolor=blue,
             anchorcolor=blue,citecolor=blue}
\usepackage{xcolor}
\usepackage[foot]{amsaddr}

\numberwithin{equation}{section}
\newtheorem{theorem}{Theorem}[section]

\newtheorem{assumption}[theorem]{Assumption}
\newtheorem{corollary}[theorem]{Corollary}

\newtheorem{definition}[theorem]{Definition}
\newtheorem{lemma}[theorem]{Lemma}

\newtheorem{proposition}[theorem]{Proposition}
\theoremstyle{remark}
\newtheorem{remark}[theorem]{Remark}

\newcommand{\one}{\mathds{1}}
\renewcommand{\div}{\text{div}}
\newcommand{\eps}{\varepsilon}
\newcommand{\E}{\mathbb{E}}

\newcommand{\R}{\mathds{R}}
\newcommand{\N}{\mathds{N}}
\renewcommand{\L}{{\mathcal L}}

\renewcommand{\P}{\mathds{P}}

\newcommand{\M}{\mathcal{M}}

\newcommand{\I}{\mathcal{I}}

\newcommand{\veps}{\varepsilon}

\DeclareMathOperator{\vol}{Vol}

\DeclareMathOperator{\diam}{diam}

\DeclareMathOperator{\Prob}{\mathbb{P}}

\DeclareMathOperator{\Lip}{Lip}
\DeclareMathOperator{\Span}{span}

\newcommand{\red}{\color{red}}

\definecolor{mygreen}{rgb}{0.1,0.75,0.2}

\newcommand{\nc}{\normalcolor}

\newcommand{\applied}[2]{\langle #1,#2\rangle}
\DeclarePairedDelimiter\norm{\lVert}{\rVert}
\DeclarePairedDelimiter\abs{\lvert}{\rvert}
\title[Spectral convergence of graph Laplacians]{Improved spectral convergence rates for graph Laplacians on $\veps$-graphs and $k$-NN graphs}
\author{Jeff Calder \and Nicol\'as Garc\'ia Trillos}
\address[J.~Calder]{School of Mathematics, University of Minnesota}
\email{jcalder@umn.edu}
\address[N.~ Garc\'{i}a Trillos]{Department of Statistics, University of Wisconsin-Madison}
\email{nicolasgarcia@stat.wisc.edu}
\thanks{JC was supported by NSF-DMS grant 1713691. NGT was supported by NSF grant DMS 1912802}

\begin{document}
\maketitle

%
\begin{abstract}
	
In this paper we improve the spectral convergence rates for graph-based approximations of Laplace-Beltrami operators constructed from random data. We utilize regularity of the continuum eigenfunctions and strong pointwise consistency results to prove that spectral convergence rates are the same as the pointwise consistency rates for graph Laplacians. In particular, for an optimal choice of the graph connectivity $\eps$, our results show that the eigenvalues and eigenvectors of the graph Laplacian converge to those of the Laplace-Beltrami operator at a rate of $O(n^{-1/(m+4)})$, up to log factors, where $m$ is the manifold dimension and $n$ is the number of vertices in the graph. \nc Our approach is general and allows us to analyze a large variety of graph constructions that include $\veps$-graphs and $k$-NN graphs. 

	
\end{abstract}

\section{Introduction}

Our work is motivated by applications in machine learning, statistics and artificial intelligence. There, the goal is to learn structure from a given data set $X=\{ x_1, \dots, x_n \}$. To do this several authors have proposed the use of graphs to endow data sets with some geometric structure, and have utilized graph Laplacians to understand how information propagates on the graph representing the data. Graph Laplacians and their spectra form the basis of algorithms for supervised learning  \cite{zhu2003semi,smola2003kernels,Belkin-Niyogi-Sindhwani,Ando}, clustering \cite{ng2002spectral,vonLux_tutorial} and dimensionality reduction~\cite{belkin2002laplacian,Coifman7426}. The works \cite{rosasco2013nonparametric,tibshirani2014adaptive,wang2016trend}  discuss Laplacian regularization  in the context of non-parametric regression. Bayesian approaches to learning where graph Laplacians are used to define covariance matrices for Gaussian priors have been proposed in  \cite{Zhu-2003,kirichenko2017,BertoStuart}.

To better understand algorithms based on graph Laplacians, it has proven useful to study the large sample size asymptotics of graph Laplacians when these capture the closeness of data points in Euclidean space, as is the case in constructions such as $\veps$-graphs or $k$-NN graphs. In this limit, we pass from discrete graph Laplacians to continuum Laplace-Beltrami operators, or weighted versions thereof, and in particular graph Laplacians are seen as specific discretizations of continuum operators. By analyzing the passage to the limit one effectively studies the consistency of algorithms that utilize said operators. In doing so, one gathers information about allowed choice of parameters, and gets insights about computational stability of algorithms (e.g. \cite{TASK,FrancaBamdad1,FrancaBamdad2}). Naturally, in order for the ``passage to the continuum" to imply any sort of consistency for a particular machine learning algorithm, it is important to study the convergence in an appropriate sense. 

Early work on consistency of graph Laplacians focused on pointwise consistency results for $\veps$-graphs (see, for example, \cite{singer2006graph,hein2005graphs,hein2007graph,belkin2005towards,ting2010analysis,GK}). There, as well as hereinafter, the data is assumed to be an \emph{i.i.d.}~sample of size $n$ from a ground truth measure $\mu$ supported on an $m$-dimensional submanifold $\M$ embedded in a high dimensional Euclidean space $\R^d$ (i.e., the \emph{manifold assumption}~\cite{ssl}) and pairs of points that are within distance $\veps$ of each other are given high weights. Pointwise consistency results show that as $n\to \infty$ and the connectivity parameter $\veps \to 0$ (at a slow enough rate), the graph Laplacian applied to a fixed smooth test function converges to a continuum operator, such as Laplace-Beltrami applied to the test function. Recent work is moving beyond pointwise consistency and studying the sequence of solutions to graph-based learning problems and their continuum limits, using tools like $\Gamma$-convergence \cite{slepcev2019analysis,dunlop2019large,trillos2016consistency,calder2019properly}, tools from PDE theory  \cite{calder2018game,calder2019consistency,calder2019properly,RyanNicolas2019,flores2019algorithms,yuan2020continuum} including the maximum principle and viscosity solutions, and more recently random walk and Martingale methods \cite{calder2020rates}. Regarding spectral convergence of graph Laplacians, the regime $n \rightarrow \infty$ and $\veps$ constant was studied in \cite{vLBeBo08}, and 
in \cite{SinWu13} which analyzes connection Laplacians. Works that have studied regimes where $\veps$ is allowed to decay to zero include \cite{trillos2018variational}, \cite{Shi2015}, \cite{BIK}, and \cite{trillos2018spectral}.

The starting point for our work is the paper  \cite{trillos2018spectral} which used ideas from \cite{BIK} in order to obtain what to the best of our knowledge are the state of the art results on spectral convergence of $\veps$-graph Laplacians. These results can be summarized as follows. With very high probability the error of approximation of eigenvalues of a continuum elliptic differential operator by the eigenvalues of their graph Laplacian counterpart scales like 
\[ \veps + \frac{\log(n)^{p_m}}{n^{1/m}\veps}, \]
where $p_m= 1/m$ for $m \geq 3$, $p_2 = 3/4$, and $\eps$ is the length scale for the graph construction. These results suggested that the best rate of convergence is achieved when $\veps$ is chosen to scale like $\sqrt{\frac{\log(n)^{p_m}}{n^{1/m}}}$, in which case the convergence rate for eigenvalues is $O(n^{-1/2m})$, up to log factors. For eigenvectors, the error of approximation in the $L^2$ norm was shown to scale like the square root of the convergence rate of eigenvalues, so $O(n^{-1/4m})$ up to log factors. In this paper, we improve in several regards the results presented in \cite{trillos2018spectral}. Our contributions to the analysis of spectral convergence of graph Laplacians constructed from $\veps$-graphs are summarized as follows:

\begin{enumerate}
\item In the $\veps$-graph setting, we show that the eigenvalues of the graph Laplacian converge (with rates), provided that $\veps$ scales like
\[ \frac{\log(n)^{1/m}}{n^{1/m}} \ll \veps \ll1. \]
This result is valid for all $m \geq 2$. This improves the results in \cite{trillos2018spectral} by removing an additional logarithmic term. In a sense, the lower bound on the allowed values for $\veps$ for the convergence to hold is an optimal requirement due to the connectivity threshold results for random geometric graphs \cite{PenroseBook}. 

\item In the $\veps$-graph case, when $\veps$ scales like
\begin{equation}\label{eq:epsregime}
 C\left( \frac{\log(n)}{n} \right)^{\frac{1}{m+4}} \leq \veps \ll 1,  
\end{equation}
we show that the rate of convergence of eigenvalues coincides with the \textit{pointwise} convergence rates of the graph Laplacian (e.g \cite{hein2007graph}), and in particular with high probability scale linearly in the connectivity length-scale $\veps$. If we choose $\eps = C\left( \frac{\log(n)}{n} \right)^{1/(m+4)}$, then we obtain convergence rates of $O(n^{-1/(m+4)})$, up to log factors, which is sharper than the $O(n^{-1/2m})$ convergence rate from \cite{trillos2018spectral} when $m \geq 5$. 

\item We establish convergence rates for eigenfunctions under $L^2$-type distances that will be made explicit later on. In particular, in the same regime for $\eps$ given in \eqref{eq:epsregime}, we establish that the rate of convergence of eigenvectors scales linearly in $\veps$, matching the convergence rate of eigenvalues as well as the pointwise convergence rates. Thus, choosing again $\eps = C\left( \frac{\log(n)}{n} \right)^{1/(m+4)}$, we obtain convergence rates for eigenvectors of $O(n^{-1/(m+4)})$,  which is far sharper than the $O(n^{-1/4m})$ convergence rates from \cite{trillos2018spectral}.
\end{enumerate}

A second main contribution of our work is to provide spectral consistency results for graph Laplacians constructed from $k$-NN graphs. Our work is the first one to  obtain any rates of convergence in such a setting.  Moreover, in proving the spectral convergence we also obtain rates for pointwise convergence which to the best of our knowledge are also new in the literature. In fact, the only other works in the literature that we are aware of that have rigorously addressed consistency for graph Laplacians associated to $k$-NN graphs are \cite{ting2010analysis} where pointwise convergence is analyzed (without providing any rates), and \cite{DBLP:journalssimodsTrillos19} where asymptotic spectral convergence is discussed, but no rates are provided. In \cite{flores2019algorithms}, pointwise consistency with rates is established for the game-theoretic $p$-Laplacian on $k$-NN graphs.

In practical applications, $k$-NN graphs are almost always preferred over $\eps$-graphs, due to their far better sparsity and connectivity properties (see, e.g.,  \cite{flores2019algorithms,calder2020poisson} for semi-supervised learning, and \cite{zelnik2005self} for spectral clustering). Since the $k$-nearest neighbor relation is \emph{not symmetric}, $k$-NN graphs are normally symmetrized in order to ensure the graph Laplacian is self-adjoint and the spectrum real-valued. On a symmetrized $k$-NN graph, the local neighborhood is no longer longer a Euclidean or geodesic ball, and is in fact not even symmetric. This raises technical difficulties in obtaining pointwise consistency results with rates, and makes the analysis far more involved than it is for $\eps$-graphs.\nc 

Our contributions in this setting are as follows:

\begin{enumerate}
\item We provide spectral convergence rates for graph Laplacians when the graph is a $k$-NN graph, provided $k$ scales like 
\[ \log(n) \ll k \ll n. \]
This result is valid for all $m \geq 2$. Moreover, we show that the rates of convergence coincide with the pointwise convergence rates (see Theorem \ref{thm:sknncon} below) when $k$ scales like
\[ Cn^\frac{4}{m+4}\log(n)^\frac{m}{m+4}\leq  k \ll n.\]

\item We establish convergence rates for eigenvectors under different topologies of interest that will be discussed later on. Moreover, in the regime
\[ Cn^\frac{4}{m+4}\log(n)^\frac{m}{m+4} \ll k \ll n,\] 
the convergence rate for eigenvectors coincides with the convergence rates of eigenvalues and also the pointwise convergence rates from Theorem \ref{thm:sknncon}.
\end{enumerate}

It is worth mentioning that all our estimates hold with high probability for finite (although possibly large) $n$. These results imply a quantitative improvement to a large body of work that has built on previous spectral convergence results.
For example the works \cite{JMLR:v20:19-261,trillos2018spectral,GeoemtricStructure} get directly benefited from our new estimates.

There are two essential steps in our analysis that allow us to improve in several regards the rates presented in \cite{trillos2018spectral} for the $\veps$-graph case. In the first step, we use a simple modification of the construction of \textit{discretization} and \textit{interpolation}  maps introduced in \cite{BIK} (a construction that was later used in \cite{trillos2018spectral}, though cast in the language of optimal transport), in order to prove spectral convergence (with rates) for a wider range of scalings of $\veps$ valid for all dimensions $m \geq 2$. A more detailed outline of the construction of these maps and a discussion on what needs to be adjusted from \cite{trillos2018spectral} is discussed in Section \ref{sec:outlineproofs} below.
 
The second step in our analysis makes use of a simple argument for comparing eigenvalues of different \emph{self-adjoint} operators. To illustrate the idea, let $A,B:H\to H$ be linear operators on a Hilbert space $H$, with $A$ self-adjoint. Let $u$ be an eigenfunction of $A$ with eigenvalue $\lambda_u$, and let $w$ be an eigenfunction of $B$ with eigenvalue $\lambda_w$. We may assume $\|u\|_H=\|w\|_H=1$. Since $A$ is self-adjoint
\[\lambda_u\langle u,w\rangle_H = \langle Au,w\rangle_H = \langle u,Aw\rangle_H = \lambda_w\langle u,w\rangle_H +\langle u,(A-B)w\rangle_H, \]
and thus
\[|\lambda_u-\lambda_w|\leq \frac{\|Aw-Bw\|_H}{|\langle u,w\rangle_H|}.\]
This inequality allows us to convert pointwise estimates on $\|Aw-Bw\|_H$ into estimates on the spectrum, provided $\langle u,w\rangle_H$ is bounded away from zero. For graph Laplacians, $A$ should, say, represent the graph Laplacian, while $B$ represents the continuum Laplace-Beltrami operator (or, more accurately, its restriction to the graph). The key ingredients in our proof are good pointwise estimates, which rely \emph{essentially} on the regularity of the continuum eigenfunctions, and the \emph{a priori} eigenfunction convergence rate from the first step of our analysis, which ensures $\langle u,w\rangle_H$ is bounded away from zero. The bottom line is that our \emph{a priori} (non-optimal) spectral convergence rates can be bootstrapped to make them coincide with the pointwise consistency rates, provided we are willing to shrink the allowed asymptotic scaling for $\veps$ slightly. The consistency of eigenfunctions will be a consequence of the a priori convergence rate for eigenvalues. This is made explicit by following some of the steps in the proof of the classical Davis-Kahan theorem.

Regarding the results for $k$-NN graphs, we first notice that these types of graphs can be thought of intuitively as $\veps$-graphs where one allows $\veps$ to vary in space. Given the inhomogeneity of the natural length scale $\veps$ (and which intuitively is influenced by data density), the first part of our analysis must rely on the definition of new discretization and interpolation maps that are tailored to the inhomogeneous length-scale setting.
After a careful analysis, we are able to provide a priori spectral convergence rates analogous to the a priori rates obtained for $\veps$-graphs. These non-optimal rates can then be bootstrapped to improve them just as in the $\veps$-graph case, using the pointwise consistency results that we derive in Theorem \ref{thm:sknncon}. We note that the pointwise consistency results for graph Laplacians on $k$-NN graphs do not follow directly from viewing the graph as an $\veps$-graph with $\veps$ varying in space. Indeed, looking forward to the proof of Theorem \ref{thm:sknncon}, the local neighborhood on a mutual (or exclusive) $k$-NN graph is asymptotically non-symmetric, due to non-uniformity of the data distribution, and so pointwise consistency for $k$-NN graph Laplacians requires a far more careful analysis than for $\veps$-graphs, where the local neighborhoods are balls. 
\nc


\subsection{Outline} 
The rest of the paper is organized as follows. In Section \ref{Setup} we give the precise set-up used throughout the paper, state our assumptions, and present our main results. Specifically, Section \ref{GraphConstructions} contains the precise definitions of the graph constructions that we study. In Section \ref{sec:eigenvalues} we state our main results regarding convergence of eigenvalues for both $\veps$-graphs as well as $k$-NN graphs, and in Section \ref{sec:eigenvectors} we present the results regarding convergence of eigenvectors. In Section \ref{sec:outlineproofs} we provide an outline of our proofs. In Section \ref{sec:consistency} we present the pointwise consistency results of graph Laplacians which will be needed later on. In Section \ref{sec:mainproofs} we present the proofs of our main results. More specifically, in Section \ref{sec:proofsveps} we present the analysis for the $\veps$-graph case, and in Section \ref{sec:proofsUnk} for $k$-NN graphs. In Section \ref{sec:TL2convergence} we discuss other modes of convergence for eigenvectors, and in particular the $TL^2$-convergence which implies Wasserstein convergence of Laplacian embeddings.

\section{Set-up and main results}
\label{Setup}

Let $\M$ be a compact, connected, orientable, smooth $m$-dimensional manifold embedded in $\R^d$. We give to $\M$ the Riemannian structure induced by the ambient space $\R^d$. With respect to the induced metric tensor, we let $\vol_\M$ be $\M$'s volume form and we let $\mu$ be a probability measure supported on $\M$ with density (with respect to the volume form) $\rho : \M \rightarrow (0,\infty)$ which we assume is bounded, and bounded away from zero, i.e.
\[ 0<\rho_{min} \leq \rho \le \rho_{max}< \infty \]
and is at least $C^{2,\alpha}(\M)$.

Let $X= \{ x_1, \dots, x_n \}$ be a set of i.i.d.~samples from $\mu$, and let $\mu_n$ denote the associated empirical measure, i.e. 
\[\mu_n := \frac{1}{n} \sum_{i=1}^n \delta_{x_i}. \]

In what follows we will use the notation $L^2(\mu)$ to denote the space of $L^2$-functions with respect to the measure $\mu$, and by $L^2(\mu_n)$ the space of functions $u: X \rightarrow \R$. We endow $L^2(\mu)$ with the inner product
\[\langle f , \tilde f \rangle_{L^2(\mu)} = \int_{\M} f(x) \tilde f(x) d\mu(x), \quad f , \tilde f \in L^2(\mu), \] 
and $L^2(\mu_n)$ with the inner product
\[ \langle u , \widetilde u \rangle_{L^2(\mu_n)} := \frac{1}{n}\sum_{i=1}^n u(x_i) \widetilde u(x_i), \quad u, \widetilde u \in L^2(\mu_n). \]

\subsection{Graph constructions}
\label{GraphConstructions}

In this section we define two different graph constructions on $X$ with the purpose of leveraging the geometry of the manifold $\M$.

\subsubsection{\texorpdfstring{$\veps$}{eps}-graphs}

Let $\veps>0$. We construct a weighted graph $G^\veps= ( X, w^\veps)$ as follows. First, we put an edge between $x_i$ and $x_j$ and between $x_j$ and $x_i$ (and write $x_i \sim x_j$)  provided that 
\[ \abs{x_i-x_j} \leq \veps,\]
where in the above, $|x_i- x_j|$ is the \textit{Euclidean} distance between the points $x_i,x_j$. We let $E=\{(i,j) \in \{1,\dots,n\}^2 : x_i\sim x_j  \}$ be the set of such edges. We may endow edges with weights that depend inversely on the distance between the vertices connected by them. For that purpose, let $\eta\colon [0,\infty) \rightarrow [0,\infty)$ be a non-increasing function with support on the interval $[0,1]$ and whose restriction to $[0,1]$ is Lipschitz continuous. 
We can assume without the loss of generality that 
\begin{equation}\label{eq:normalizedkernel}
\int_{\R^m} \eta(\abs{x}) dx =1.
\end{equation}
For convenience we assume that $\eta(1/2)>0$.
We introduce the constant
\begin{equation} \label{def:sigma}
\sigma_\eta \coloneqq \int_{\R^m} \abs{y_1}^2 \eta(\abs{y}) dy,
\end{equation}
where $y_1$ represents the first coordinate of the vector $y\in \R^m$. Notice that a simple computation using radial coordinates shows that when $\eta(t) = \mathds{1}_{[0,1]}(t)$ then $ \sigma_\eta = \frac{\alpha_m}{m+2}$, where $\alpha_m$ is the volume of the $m$-dimensional Euclidean unit ball.

To every  given edge $(i,j) \in E$ we assign the weight $w^\veps_{x_ix_j}$ where
\begin{equation}\label{eqn:weights}
w^\eps_{xy} = \eta\left( \frac{|x - y|}{\eps} \right).
\end{equation}
and we can consider the weighted graph $G^\eps(X,w^\eps)$. In fact, note that if the points $x_i$, $x_j$ are not connected by an edge in $E$ then $w^\veps_{x_ix_j}=0$.   

Having introduced the graph $G^\veps$, we define an associated graph Laplacian operator $\L^\veps$ which for a given $u \in L^2(\mu_n)$ is defined as
\begin{equation}\label{eq:gL}
\L^\eps u(x) = \frac{1}{n\eps^{m+2}}\sum_{j=1}^n w^\eps_{x_jx}(u(x) - u(x_j)).
\end{equation}
Note that in principle $x$ need not be a vertex of the graph to make sense of the above expression, but unless otherwise stated, in the sequel $\L^\veps$ will be thought of as an operator $\L^\veps : L^2(\mu_n) \rightarrow L^2(\mu_n)$.

It is well known in the literature (e.g. \cite{vonLux_tutorial}) that $\L^\veps$ is a positive definite self-adjoint operator with respect to the inner product $\langle  \cdot , \cdot \rangle_{L^2(\mu_n)}$. In particular, we can list its eigenvalues (repeated according to multiplicity) as
\[ 0 \leq \lambda_{1}^\veps \leq \lambda_2^\veps \leq \dots \leq  \lambda_n^\veps ,\]
$\lambda_1^\veps$ always being equal to zero. Also, it is possible to find an orthonormal basis of eigenvectors for $L^2(\mu_n)$. Moreover, a \textit{graph Dirichlet energy} associated to  $\L^\veps$ defined by
\begin{equation}
 b_\veps(u):= \frac{1}{n\eps^{m+2}}\sum_{i,j} w^\eps_{x_ix_j} ( u(x_i) - u(x_j) )^2 = 2 \langle  \L^\veps  u , u \rangle_{L^2(\mu_n)}, \quad u \in L^2(\mu_n).
\end{equation}
can be used to define the eigenvalues of $\L^\veps$ variationally. Namely, the Courant-Fisher minmax principle says that
\begin{equation}
\lambda_l^\veps =  \frac{1}{2}\min_{S \in \mathfrak{S}_l}\max_{u \in S \setminus \{ 0 \}} \frac{b_\veps(u) }{\lVert u \rVert_{L^2(\mu_n)}^2} 
\label{eqn:varcharacveps}
\end{equation}
where in the above $\mathfrak{S}_l$ denotes the set of all linear subspaces of $L^2(\mu_n)$ of dimension $l$.

In this paper we will restrict our attention to the definition of graph Laplacian in \eqref{eq:gL} (known in the literature as unnormalized Laplacian), but we note that there are several  other normalizations that are of high interest in machine learning (e.g. \cite{vonLux_tutorial}), and we expect we can carry out a similar analysis for other normalizations.

\nc

\subsubsection{Undirected $k$-NN graph}

A different graph construction on $X$ proceeds not by fixing a length-scale $\veps$ but rather by specifying for each point in $X$ a set of \textit{nearest neighbors}. 
\begin{definition}
Let $k \in \N$.	We define a relation $\sim_k$ on $X \times X$ by declaring
\[x_i \sim_k x_j , \]
if $x_j$ is among the $k$ nearest neighbors (in the Euclidean distance sense) of $x_i$.
\end{definition}
In this section we symmetrize the relation $\sim_k$ and place an edge between $x_i$ and $x_j$ if $x_i\sim_k x_j$ or $x_j \sim_k x_i$. This is often called the \emph{symmetric} $k$-nearest neighbor (or $k$-NN) graph~\cite{maier2009optimal}. Another popular construction is the \emph{mutual} $k$-NN graph, where we connect $x_i$ to $x_j$ if $x_i\sim_k x_j$ \emph{and} $x_j \sim_k x_i$. While in the sequel we formulate our results for symmetric $k$-NN graphs, they apply with minor modifications to mutual $k$-NN graphs as well.

To construct the $k$-NN graph Laplacian, let 
\begin{equation}\label{eq:Ne}
N_\eps(x) = \sum_{j\, : \,0<|x_j-x|\leq \eps} 1
\end{equation}
be the number of random samples in a punctured Euclidean $\eps$-neighborhood of $x$. Given $1 \leq k\leq n-1$ define
\begin{equation}\label{eq:epsk}
\eps_k(x) := \min \{ \eps > 0 \, : \, N_\eps(x) \geq k\}.
\end{equation}
The value $\eps_k(x)$ is the Euclidean distance from $x$ to the $k^{\rm th}$ nearest neighbor of $x$ from the samples $x_1,\dots,x_n$. Thus, $x_i\sim_k x_j$ if and only if $|x_i-x_j|\leq \eps_k(x_i)$. Finally, we define
\begin{equation}\label{eq:rk}
r_k(x,y): = \max\{\eps_k(x),\eps_k(y)\}.
\end{equation}
Notice that $|x_i-x_j|\leq r_k(x_i,x_j)$ if and only if $x_i$ and $x_j$ are connected by an edge in the symmetric $k$-NN graph. We would obtain the mutual $k$-NN graph by setting $r_k(x,y) = \min\{\eps_k(x),\eps_k(y)\}$.

The undirected $k$-NN graph Laplacian of an element $u \in L^2(\mu_n)$ is defined by
\begin{equation}\label{eq:sknnL}
\L^{k}u(x)= \frac{1}{n}\left( \frac{n\alpha_m}{k} \right)^{1+2/m}\sum_{j=1}^n w_{x_jx}^{r_k(x_j,x)}(u(x) - u(x_j)),
\end{equation}
where $w_{x_jx}^{r_k(x_j, x)}$ has the same meaning as in \eqref{eqn:weights}. As in the $\veps$-graph case, we will assume for the most part that $x$ is one of the elements in $X$ and hence $\L^{k}$ is interpreted as an operator from $L^2(\mu_n)$ into $L^2(\mu_n)$. We list the eigenvalues (repeated according to multiplicity) of $\L^k$ as
\[ 0 \leq \lambda_{1}^k \leq \lambda_2^k \leq \dots \leq  \lambda_n^k ,\]
and define an associated graph Dirichlet energy by 
\begin{equation}
b_k(u):= \frac{1}{n^2} \left( \frac{n\alpha_m}{k} \right)^{1+2/m} \sum_{i,j} w^{r_k(x_i, x_j)}_{x_ix_j} ( u(x_i) - u(x_j) )^2 = 2 \langle  \L^\veps  u , u \rangle_{L^2(\mu_n)}, \quad u \in L^2(\mu_n).
\label{def:bkNN}
\end{equation}
\nc
This functional can be used to define the eigenvalues of $\L^k$ via the variational formula
\begin{equation}
\lambda_l^k = \frac{1}{2} \min_{S \in \mathfrak{S}_l}\max_{u \in S\setminus \{0 \}} \frac{b_k(u) }{\lVert u \rVert_{L^2(\mu_n)}^2}. 
\end{equation}

\begin{remark}
Notice that the rescaling factor $\left(\frac{n \alpha_m}{k}\right)^{1+2/m}$ in \eqref{eq:sknnL} is equal to $1/r^{m+2}$ where $r$ is the radius of an $m$-dimensional Euclidean ball with volume $k/n$. This is the same type of rescaling factor that appears in the definition of the $\eps$-graph Laplacian $\L^\veps$ in \eqref{eq:gL}. 
\end{remark}

\subsection{Convergence of eigenvalues}
\label{sec:eigenvalues}

\subsubsection{$\veps$-graph}

In our first main result we establish error bounds between the eigenvalues $\lambda_{l}^\veps$ of the graph Laplacian $\L^\veps$, and the eigenvalues $\lambda_l$ of a differential operator $\Delta_\rho$ that for smooth functions $f: \M \rightarrow \R$ is defined as
\begin{equation}\label{eq:wlb}
\Delta_\rho f := - \frac{1}{2\rho} \div(\rho^2\nabla f).
\end{equation}
In the above $\div$ stands for the divergence operator on $\M$, and $\nabla$ for the gradient in $\M$. It turns out that $\Delta_\rho$ is a positive semi definite operator with a nice point spectrum so that in particular its eigenvalues can be listed (repeated according to multiplicity) in increasing order as
\[ 0 \leq  \lambda_1 \leq \lambda_2 \leq \dots \]
In addition, each eigenvalue has finite multiplicity, the sequence of eigenvalues grows to infinity, and it is possible to find an orthonormal basis for $L^2(\mu)$ consisting of eigenfunctions, i.e. functions $f$ that satisfy the equation $\Delta_\rho f =\lambda f$.

The eigenvalues of $\Delta_\rho$ can be described variationally in terms of the Dirichlet energy
\[    D_2(f) :=  \begin{cases}  \int_{\M}  \lvert \nabla f (x) \rvert^2 \rho^2(x) d\vol_\M(x),  & \text{ if } f \in H^1(\mu) \\  +\infty, & \text{ otherwise, }  \end{cases} \]   
where in the above $H^1(\mu)$ is the space of functions with a weak gradient in $L^2(\mu)$. Indeed, 
\begin{align}
\label{eqn:varcharaccont}
\lambda_l = \frac{1}{2}\min_{S \in \mathfrak{S}_l} \max_{f\in {S}\setminus\{0\}} \frac{D_2(f)}{\norm{f}^2_{L^2(\mu)}}
\end{align}
where in the above $\mathfrak{S}_l$ is the set of linear subspaces of $L^2(\mu)$ of dimension $l$. The minimum is reached when $S$ is taken to be the span of the first $l$ eigenfunctions of $\Delta_\rho$. We also notice that when $f$ is regular enough (in particular if $ f \in H^2(\mu)$) then $D_2(f)$ coincides with $2\langle \Delta_\rho  f , f \rangle_{L^2(\mu)}$ thanks to the integration by parts formula.

Throughout the paper we will state our probabilistic estimates in terms of parameters $\veps$ (the graph connectivity), $n$ (the number of data points) and two parameters $\theta, \widetilde \delta$ that describe an error of discrete-to-continuum approximation (see Proposition \ref{prop:AuxiliaryDensity} for more details on the meaning of these two parameters). The following are technical ``smallness" assumptions on our parameters that guarantee that we have entered the regime where the statements of our theorems are meaningful.

\begin{assumption}
In the $\veps$-graph setting, for $\theta, \widetilde \delta$ and $\veps$ we assume 
\begin{enumerate}
\item $\veps$ is small enough and in particular satisfies 
\begin{equation}\label{eq:eps}
2\eps < \min\{1,i_0,K^{-1/2},R/2\}=:2\eps_\M.
\end{equation} 
where $i_0, K, R$ are geometric quantities defined in Section \ref{sec:diffgeom}.
\item $\widetilde \delta \leq \frac{1}{4}\veps$.
\item $\widetilde \delta$ is larger than $\frac{1}{n^{1/m}}$.
\item $ C(\theta + \widetilde \delta) \leq \frac{\rho_{min}}{2}$.
\end{enumerate}

In the $k$-NN setting, for $\theta, \widetilde \delta$ and $k$ we assume
\begin{enumerate}
	\item $(k/n)^{1/m}$ is small enough and in particular satisfies 
	\[2(k/n)^{1/m} < C_\rho\min\{1,i_0,K^{-1/2},R/2\}=:2\eps_\M.\]
	where  $C_\rho$ is a constant that depends on the density $\rho$.
	\item $\widetilde \delta \leq c_\rho (k/n)^{1/m} $.
	\item $\widetilde \delta$ is larger than $\frac{1}{n^{1/m}}$.
	\item $ C(\theta + \widetilde \delta) \leq \frac{\rho_{min}}{2}$.

\end{enumerate}
\label{assumptions}
\end{assumption}

\nc

We are ready to state our first main result.

\begin{theorem}[Rate of convergence for eigenvalues]
	\label{thm:evaluesveps}
	Suppose that $\M, \mu$ satisfy the assumptions from Section \ref{Setup}. Suppose that the quantities $\widetilde \delta , \theta, \veps$ satisfy Assumptions \ref{assumptions}. Then:
	\begin{enumerate}
	 \item  For every $l\in \N$, there is a constant $c_l$ such that if
	 \[\sqrt{\lambda_l}  \veps  + C(\theta + \widetilde \delta )  \leq c_l, \]
	 then with probability at least $1- C n \exp(- Cn \theta^2 \widetilde \delta^m)$,   \nc
	\[  |\lambda_l ^\veps - \sigma_\eta\lambda_l| \leq      C\left(\veps (\sqrt{\lambda_l}+1) +  \frac{\widetilde \delta}{\veps}   +\theta    \right) \lambda_l.\] 
	\item Moreover, for every  $l \in \N$ there is a constant $c_l$ such that if 
	\[\left(\veps \sqrt{\lambda_l} +  \frac{\widetilde \delta}{\veps} +\veps   +\theta    \right) \lambda_l\leq c_l\] 
	then, with probability greater than $1-  2Cn\exp\left(-c n\eps^{m+4}\right) - Cn\exp(- cn\theta^2 \widetilde{\delta}^m )$ we have
	\[  \lvert \lambda_l^\veps -\sigma_\eta \lambda_l\rvert \leq C\veps.\] 
		\nc
	\end{enumerate}
%
\end{theorem}

\nc

Let us pause for a moment and discuss the content of Theorem \ref{thm:evaluesveps}. A consequence of the first part of the theorem is that as long as $\veps \gg  \left(\frac{\log(n)}{n}\right)^{1/m}$, then it is possible to pick $\widetilde \delta $ and $\theta$ according to $\theta = \sqrt{\frac{(\alpha+1) \log(n)}{Cn\widetilde \delta^m} }$ for some large enough $\alpha$ in such a way that as $n$ grows, both $\theta$ and $\widetilde\delta $ converge to zero, and with probability one
\[ |\lambda_l^\veps - \sigma_\eta \lambda_l  | \leq C(\widetilde \delta + \theta) \rightarrow 0.\]
That is, we can make the error of approximation of eigenvalues converge to zero as $n \rightarrow \infty$ provided that $\veps\gg \left(\frac{\log(n)}{n}\right)^{1/m}$. This result is valid for all $m \geq 2$, removing in this way an extra logarithmic factor that was present in the results from \cite{trillos2018spectral} when the dimension was $m=2$. Of course, the first part of the theorem gives more than just asymptotic convergence and we do get rates. We would like to note that the removal of this logarithmic factor has been done recently in \cite{MullerPenrose} and \cite{MShaCCS} for other, related, variational problems on graphs. Here we remove the extra logarithmic factors while obtaining rates of convergence. Notice, however, that our estimates would predict that the rates of convergence become quite slow as we get close to the scaling $\left(\frac{\log(n)}{n} \right)^{1/m}$, which is the connectivity threshold of the graph. The rates of convergence degrade as we approach the connectivity threshold because the graph is irregular on the smallest length scale of nearest neighbors. However, we expect that on a larger macroscopic scale, the irregularities will average out, and it may be possible to extend the linear $O(\eps)$ rate to the range $\left(\frac{\log(n)}{n} \right)^{1/m} \ll \eps \ll \left(\frac{\log(n)}{n} \right)^{1/(m+4)}$. We expect this to require deeper tools and techniques from the theory of stochastic homogenization, and leave this to future work. \nc

The second part of the theorem in particular implies that if $\veps$ is larger than $\left( \frac{\log(n)}{n} \right)^{\frac{1}{m+4}}$, then with high probability the rate of convergence of eigenvalues and eigenvectors of $\L^\veps$ towards those of $\Delta_\rho$ is linear in $\veps$. We believe that for eigenvectors the linear convergence rate is the best possible for a generic manifold $\M$. In numerical analysis of PDEs, one normally sees second order $O(\eps^2)$ convergence rates for the spectrum of the Laplacian on a uniform grid in Euclidean space, due to using \emph{symmetric} stencils for the Laplacian (see, e.g., \cite{leveque2007finite}), where first order error terms exactly cancel out, and the lack of discretization errors due to the curvature of the manifold. In our manifold setting, the curvature of the manifold arises in the first order error terms and prevents the scheme from attaining second order accuracy. We also note that the numerical scheme provided by the graph Laplacian \eqref{eq:gL} is not a symmetric stencil (i.e., the neighbors of $x$ are not of the form $x+v$ and $x-v$), and so the first order terms cancel out only on average, after random fluctuations are accounted for, which necessitates the larger length scale $\eps \gg\left(\frac{\log(n)}{n} \right)^{1/(m+2)}$ for pointwise consistency. In the flat Euclidean setting, it may be possible to improve convergence rates to $O(\eps^2)$ with a more severe length scale restriction $\eps \gg\left(\frac{\log(n)}{n} \right)^{1/(m+6)}$.  Regarding convergence rates for eigenvalues, we actually believe that they may be faster than the convergence rates for eigenfunctions. This is because eigenvalues may be defined by computing the Dirichlet energy of a normalized eigenvector of the graph Laplacian, and thus is computed by taking a double sum which in principle may homogenize better. This is by no means a proof and only reveals a simple intuition.

Finally, we would like to remark that we can trace the dependence of all the constants appearing in Theorem \ref{thm:evaluesveps} on the different parameters of the problem, but we have decided not state them explicitly to immensely facilitate our presentation.

\nc

\subsubsection{$k$-NN graph}

In the $k$-NN graph case we show that the eigenvalues of $\L^k$ converge towards the eigenvalues of the operator $\Delta_\rho^{NN}$ defined for smooth functions $f : \M \rightarrow \R$ according to
\begin{equation}\label{eq:Dp3}
\Delta_{\rho}^{NN} f :=- \frac{1}{2\rho}\text{div}(\rho^{1-2/m}\nabla f).
\end{equation}
 This operator has similar spectral properties as $\Delta_\rho$. In particular, if we list its eigenvalues as 
\[ 0 \leq \lambda_1 \leq \lambda_2\leq \dots ,\]
they can be written as
\[\lambda_l =\frac{1}{2} \min_{S \in \mathfrak{S}_l} \max_{f\in {S}\setminus\{0\}} \frac{D_{1-2/m}(f)}{\norm{f}^2_{L^2(\mu)}}\]
where now the associated Dirichlet energy $D_{1-2/m}$ takes the form
\begin{equation}
D_{1-2/m}(f) := \begin{cases} \int_\M \lvert \nabla f(x) \rvert^2 \rho^{1-2/m}(x) d\vol_\M(x) & \text{ if } f \in H^1(\mu)\\
+\infty & \text{ otherwise. } 
\end{cases}
\label{WeightedDirichlet2}
\end{equation}

Following a remark in \cite{DBLP:journalssimodsTrillos19}, it is possible to show that the spectrum of $\L^{k}$ converges as $n \rightarrow \infty$ towards the spectrum of $\Delta_\rho^{NN}$ provided that $k$ scales like
\[ (\log(n))^{m p_m} \ll k \ll n. \]
No rates are provided in \cite{trillos2018variational} as the convergence is deduced using the notion of $\Gamma$-convergence.  Besides enlarging the set of admissible values of $k$ for which we can prove the spectral convergence, we establish the following convergence rates. They are analogous to the ones stated in Theorem \ref{thm:evaluesveps} for $\veps$-graphs.
 
\begin{theorem}[Rate of convergence for eigenvalues]
	\label{thm:evaluesks}
	Suppose that $\M, \mu$ satisfy the assumptions from Section \ref{Setup}. Suppose that the quantities $\widetilde \delta , \theta, k$ satisfy Assumptions \ref{assumptions}. Then:
	\begin{enumerate}
		\item For every $l\in \N$, there is a constant $c_l$ such that if
		\[\sqrt{\lambda_l} (k/n)^{1/m}  + C(\theta + \widetilde \delta )  \leq c_l, \] 
		then, with probability at least $1- C n \exp(- Cn \theta^2 \widetilde \delta^m)$, we have
		\[  |\lambda_l ^k - \sigma_\eta\lambda_l| \leq   +   C \left (\left(\frac{k}{n} \right)^{1/m} (\sqrt{\lambda_l}+1) +  \widetilde \delta\left(\frac{n}{k}\right)^{1/m}   +\theta  \right  ) \lambda_l.\] 
		\item Moreover, for every  $l \in \N$ there is a constant $c_l$ such that if 
		\[\left(\left(\frac{k}{n}\right)^{1/m} \sqrt{\lambda_l} +  \frac{\widetilde \delta}{\veps}\left(\frac{n}{k}\right)^{1/m}   +\theta   \right ) \lambda_l\leq   c_l\] 
		then, with probability greater than $1- Cn\exp\left( -c(k/n)^{4/m}k\right) - Cn\exp(- Cn\theta^2 \widetilde{\delta}^m )$ we have
		\[  \lvert \lambda_l^k -\sigma_\eta\lambda_l\rvert \leq C(k/n)^{1/m}.\] 
		\nc
	\end{enumerate}
\end{theorem}

\nc

\subsection{Convergence of eigenvectors}

\label{sec:eigenvectors}

We also establish convergence rates for the eigenvectors of $\L^\veps$ towards eigenfunctions of $\Delta_\rho$, as well as convergence rates of eigenvectors of $\L^k$ towards eigenfunctions of $\Delta_\rho^{NN}$. Naturally, we first need to specify the sense in which the convergence is to take place. To describe this, let us first notice that the eigenfunctions of both $\Delta_\rho$ and $\Delta_\rho^{NN}$ are continuous functions, so that pointwise evaluation of eigenfunctions makes perfect sense. In particular, we can restrict a given eigenfunction $f$ to $X$ and the quantity $\lVert u- f \rVert^2_{L^2(\mu_n)}$ would be well defined for every $u \in L^2(\mu_n)$. Furthermore, if $f$ solves the elliptic equation 
\[ \Delta_\rho f = \lambda f \]
for some $\lambda>0$, then, by the regularity theory of elliptic operators \cite{gilbarg2015elliptic} (which can be lifted from the Euclidean setting to the curved manifold setting given the assumptions that we made on $\M$ at the beginning of Section \ref{Setup}) and the regularity  assumption on the density $\rho\in C^2(\M)$, it follows that $f$ is at least a $C^3(\M)$ function. That is, its first, second and third derivatives are continuous, and given the compactness of $\M$ also bounded. The same regularity holds for eigenfunctions of $\Delta_\rho^{NN}$. Said regularity will be needed in order to apply the pointwise consistency results from Theorem \ref{thm:con1} (for the $\veps$-graph case) and Theorem \ref{thm:sknncon} (for the $k$-NN case) to a given eigenfunction $f$.

\nc

We have the following convergence result for $\veps$-graphs.

\begin{theorem}[Rate of convergence for eigenvectors]
	\label{thm:evectorsveps}
	
	Suppose that $\M, \mu$ satisfy the assumptions from Section \ref{Setup}. Suppose that the quantities $\widetilde \delta , \theta, \veps$ satisfy Assumptions \ref{assumptions}. Then:

	\begin{enumerate}
	
	\item  For every  $l \in \N$ there is a constant $c_l$ such that if 
	\[\left(\veps \sqrt{\lambda_l} +  \frac{\widetilde \delta}{\veps} +\veps   +\theta    \right) \lambda_l\leq c_l,\]
	then with probability at least $1-Cn \exp(-Cn\theta^2 \widetilde \delta^m)$, for every $u_l$ normalized eigenvector of $\L^\veps$ with eigenvalue $\lambda_l^\veps$, there is a normalized eigenfunction $f_l$ of $\Delta_\rho$ with eigenvalue $\lambda_l$ such that
	\[\lVert u_l - f_l \rVert_{L^2(\mu_n)}\leq  \left(C_l\left(  \frac{\widetilde \delta}{\veps} + \veps + \theta \right) \right)^{1/2} + C_l \widetilde \delta .\]
	
	\item For every  $l \in \N$ there is a constant $c_l$ such that if 
	\[\left(\veps \sqrt{\lambda_l} +  \frac{\widetilde \delta}{\veps} +\veps   +\theta    \right) \lambda_l\leq c_l,\]
	then with probability at least $1- Cn \exp(-C n \theta^2 \widetilde {\delta}^m) - Cn\exp\left(-C n\eps^{m+4}\right) $, for every $u_l$ normalized eigenvector of $\L^\veps$ with eigenvalue $\lambda_l^\veps$, there is a normalized eigenfunction $f_l$ of $\Delta_\rho$ with eigenvalue $\lambda_l$ such that 
	\[  \lVert u_l - f_l \rVert_{L^2(\mu_n)} \leq C_l\veps. \]
	\end{enumerate}
	
\end{theorem}

\begin{remark}
The constant $C_l$ in Theorem \ref{thm:evectorsveps} (ii) depends on the $C^3$ norm of all normalized eigenfunctions $f_l$ of $\Delta_\rho$ with eigenvalue $\lambda_l$, since our proof uses the pointwise consistency of graph Laplacians. The constant $C_l$ is also inversely proportional to the  spectral gap $\gamma_l$ defined as
\[\gamma_l = \min\{|\lambda_l - \lambda_j| \, :\, \lambda_l\neq \lambda_j\}.\]
A similar remark holds for the analogous result for $k$-nearest neighbor graphs given in Theorem \ref{thm:evectorsks}.
\label{rem:constants}
\end{remark}
\nc

As in Theorem \ref{thm:evaluesveps},  the main difference between the error estimates presented in Theorem \ref{thm:evectorsveps} is the regimes of $\veps$ for which they are meaningful. In general we may use energy arguments as in \cite{BIK} to obtain convergence rates that scale like the square root of the rates of convergence for eigenvalues.  On the other hand, in the regime $ \veps \gg \frac{\log(n)}{n^{1/{m+4}}}$ we can show that eigenfunctions of the graph Laplacian converge to eigenfunctions of $\Delta_{\rho}$ linearly in the connectivity length-scale $\veps$, coinciding in this way with the rates for eigenvalues and also with the rates of pointwise convergence (see Theorem \ref{thm:con1}).

The analogous estimates for $k$-NN graphs are the following.

\begin{theorem}[Rate of convergence for eigenvectors]
	\label{thm:evectorsks}

	Suppose that $\M, \mu$ satisfy the assumptions from Section \ref{Setup}. Suppose that the quantities $\widetilde \delta , \theta, k$ satisfy Assumptions \ref{assumptions}. Then:
\begin{enumerate}
\item  For every  $l \in \N$ there is a constant $c_l$ such that if 
\[( (k/n)^{1/m} \sqrt{\lambda_l} +  \widetilde \delta(n/k)^{1/m}   +\theta ) \lambda_l\leq c_l,\]
then with probability at least $1-Cn \exp(-Cn\theta^2 \widetilde \delta^m)$, for every $u_l$ normalized eigenvector of $\L^k$ with eigenvalue $\lambda_l^k$, there is a normalized eigenfunction $f_l$ of $\Delta_\rho^{NN}$ with eigenvalue $\lambda_l$ such that
\[\lVert u_l - f_l \rVert_{L^2(\mu_n)}\leq  \left(C_l\left(  \widetilde \delta(n/k)^{1/m} + (k/n)^{1/m} + \theta \right) \right)^{1/2} + C_l \widetilde \delta .\]

\item For every  $l \in \N$ there is a constant $c_l$ such that if 
\[((k/n)^{1/m} \sqrt{\lambda_l} +  \widetilde \delta(n/k)^{1/m}   +\theta    ) \lambda_l\leq c_l,\]
then with probability at least $1- Cn \exp(-C n \theta^2 \widetilde {\delta}^m) - Cn\exp\left( -c(k/n)^{4/m}k\right) $, for every $u_l$ normalized eigenvector of $\L^k$ with eigenvalue $\lambda_l^k$, there is a normalized eigenfunction $f_l$ of $\Delta_\rho^{NN}$ with eigenvalue $\lambda_l$ such that 
\[  \lVert u_l - f_l \rVert_{L^2(\mu_n)} \leq C_l(k/n)^{1/m}. \]
\end{enumerate}	
\end{theorem}
\nc

There are other ways to study the convergence of eigenvectors of graph Laplacians. As we will see in Section \ref{sec:TL2convergence}, Theorems \ref{thm:evectorsveps} and \ref{thm:evectorsks}, and the regularity of eigenfunctions of $\Delta_\rho$ and $\Delta_\rho^{NN}$ imply the convergence of eigenvectors in the so called $TL^2$-sense (see \cite{trillos2018spectral}). This notion of convergence makes explicit a way to compare different spectral embeddings that form the basis of algorithms like spectral clustering \cite{vonLux_tutorial}. To be more precise let $L \in \N$. With the first $L$ eigenvectors $u_1, \dots, u_L$ of $\L^\veps$ (or $\L^k$) one can construct an embedding of the data set
\[ F_n : \{ x_1, \dots, x_n \} \rightarrow \R^L  \]
\[  F_n(x_i) = \left(\begin{matrix} u_1(x_i) \\ \vdots \\ u_L(x_i) \end{matrix} \right). \]
The resulting set of points can be represented by the measure $F_{n \sharp} \mu_n$ , i.e., the push-forward of the original empirical measure $\mu_n$ by the graph Laplacian embedding $F_n$. A natural question to ask in this setting is: how far apart is $F_{n \sharp} \mu_n$  from the measure $F_{\sharp }\mu$, where $F$ is the continuum Laplacian embedding:
\[ F : \M \rightarrow \R^L \]
\[  F(x) = \left(\begin{matrix} f_1(x) \\ \vdots \\ f_L(x) \end{matrix} \right),\]
constructed from eigenfunctions $f_1, \dots, f_L$ of the continuum operator $\Delta_\rho$ (or $\Delta_\rho^{NN}$)? We answer this question in terms of the Wasserstein distance $W_2$ between $F_{n \sharp} \mu_n$ and $F_{\sharp} \mu$, which we recall is defined by
\begin{equation}
 W_2(F_{\sharp } \mu, F_{\sharp n} \mu_n )  := \min_{\pi \in \Gamma(F_{\sharp } \mu, F_{n\sharp } \mu_n)} \int_{\R^L \times \R^L} |x-y|^2 d \pi(x,y),
\label{def:Wassdistance}
\end{equation}
where $\Gamma(F_{\sharp } \mu, F_{n\sharp } \mu_n)$ denotes the set of couplings or transport plans between $F_{\sharp } \mu$ and $F_{n\sharp } \mu_n$. This way of comparing the spectral embeddings was proposed in \cite{GeoemtricStructure}.

\begin{corollary}
	\label{cor:WassersteinIneq}
	Suppose that $\M, \mu$ satisfy the assumptions from Section \ref{Setup}. 
	
	\begin{enumerate}
		\item ($\veps$-setting) Suppose that the quantities $\widetilde \delta , \theta, \veps$ are small enough as in 2 in Theorem \ref{thm:evectorsveps}. Then, with probability at least $1- Cn \exp(-C n \theta^2 \widetilde {\delta}^m) - Cn\exp\left(-C n\eps^{m+4}\right) $, we have
		\[ W_2(F_\sharp \mu, F_{n \sharp} \mu_n) \leq C \veps + C W_2(\mu, \mu_n),\]
		where $F_n$ is the graph Laplacian embedding constructed from the first $L$ eigenvectors of $\L^\veps$, and $F$ is the Laplacian embedding constructed using the eigenfunctions $f_1, \dots, f_L$ of $\Delta_\rho$ from Theorem \ref{thm:evectorsveps}.
		\item Suppose that the quantities $\widetilde \delta , \theta, (k/n)^{1/m}$ are small enough as in 2 in Theorem \ref{thm:evectorsks}. Then, with probability at least $1- Cn \exp(-C n \theta^2 \widetilde {\delta}^m) - Cn\exp\left(-C (k/n)^{4/m}k\right) $, we have
		\[ W_2(F_\sharp \mu, F_{n \sharp} \mu_n) \leq C (k/n)^{1/m} + C W_2(\mu, \mu_n),\]
		where $F_n$ is the graph Laplacian embedding constructed from the first $L$ eigenvectors of $\L^k$, and $F$ is the Laplacian embedding constructed using the eigenfunctions $f_1, \dots, f_L$ of $\Delta_\rho^{NN}$ from Theorem \ref{thm:evectorsks}.
	\end{enumerate}
\end{corollary}
\nc

\begin{remark}
Probabilistic bounds for 
\[W_2^2(\mu , \mu_n) = \min _{\pi\in \Gamma(\gamma, \widetilde \gamma)} \int_{\M \times \M}d_\M(x,y)^2 d \pi(x,y) \]
can be easily derived using a localization argument (to deal with the curved manifold) and the concentration inequalities estimating the Wasserstein distance between empirical measures and their ground truth counterparts in the Euclidean setting (see \cite{Fournier} and a more recent treatment with improved constants and better scalability with respect to dimension in \cite{JingLei}). The localization argument is used simply to partition the manifold $\M$ into a fixed number of regions that are bi-Lipschitz homeomorphic to a bounded domain in $\R^m$. One can then lift the concentration results in Euclidean space $\R^m$ to the manifold. In particular, $W_2(\mu , \mu_n)$ scales like $ \frac{1}{n^{1/m}}$.

\end{remark}
\nc

\subsection{Outline of proofs and discussion}
\label{sec:outlineproofs}

The proofs of our main theorems follow the same structure in both the $\veps$-graph and $k$-NN settings. For this reason we outline our proofs only in the $\veps$-graph setting and then simply state the modifications needed in order to cover the $k$-NN case.

Our first step is to establish some a priori (and non-optimal) rates of convergence for eigenvalues of $\L^\veps$ towards eigenvalues of $\Delta_\rho$. In principle, we could cite the results already proved in \cite{trillos2018spectral}, but we have decided to prove our own a priori results, so that we have the opportunity to present a simpler approach that works for a larger set of values of $\veps$ than those covered in \cite{trillos2018spectral}.  \nc

To discuss our construction, it is worth first reviewing the one in \cite{BIK}, which is a main influence for the construction in our work and the work in \cite{trillos2018spectral}. There the main idea is to construct maps
\[ P\colon L^2(\mu) \to L^2(\mu_n), \quad \I \colon L^2(\mu_n) \rightarrow L^2(\mu)  \]
that are almost isometries when restricted to functions of low Dirichlet energy (in discrete and continuum settings), and for which one has estimates of the form
\[ b_\veps(P f)  \leq (1+ e) \sigma_\eta D_2(f), \quad \forall f \in L^2( \mu), \]
and
\begin{equation}
\sigma_\eta D_2(\I u) \leq (1+e)b_\veps(u), \quad  \forall u \in L^2(\mu_n), 
\label{ineq:energyestimate}
\end{equation}
where $e$ is thought of as a small error term. When combined with the variational identities \eqref{eqn:varcharacveps} and \eqref{eqn:varcharaccont}, these inequalities produce error bounds for the difference between eigenvalues of $\L^\veps$ and $\Delta_\rho$: one can upper bound $\lambda_l^\veps$ with $\lambda_l + Ce$ choosing $S$ in \eqref{eqn:varcharaccont} to be the image under $P$ of the span of the first $l$ eigenvectors of $\Delta_\rho$, and likewise, one can upper bound  $\lambda_l $ by $\lambda_l^\veps + Ce$ choosing $S$ in \eqref{eqn:varcharacveps} to be the image under $\I$ of the span of the first $l$ eigenvectors of $\L^\veps$. Now, the maps $P$ and $\mathcal{I}$ introduced in \cite{BIK} are based on an $\infty$-optimal transport map between $\mu$ and $\mu_n$. That is, they use a map  $T : \M \rightarrow \{x_1, \dots, x_n \}$ that pushes forward the measure $\mu$ into the empirical measure $\mu_n$, and does so in such a way that it minimizes the quantity
\[  \delta := \sup_{x \in \M} d_\M(x, T(x)), \]  
over all possible such transport maps.  $T$ can be used to generate a tessellation $U_1, \dots, U_n$ of $\M$ (here $U_i := T^{-1}( \{ x_i \})$) , with the property that all cells $U_i$ have $\mu$-measure equal to $1/n$ and diameter bounded by $2 \delta$. The maps $P$ and $\I$ are then defined according to
\begin{equation*}
( Pf)(x_i) \coloneqq n \cdot \int_{ U_i} f(x)  \rho (x) dx, \quad f \in  L^2(  \mu).
\end{equation*}
and
\begin{equation*}
Iu \coloneqq \Lambda_{\veps - 2\delta} P^*u, \quad u \in  L^2(  \mu_n).
\end{equation*}
where $P^*$ is the adjoint of $P$ (i.e. composition with $T$) and $\Lambda_{\veps - 2\delta}$ is a convolution operator with respect to a very carefully chosen kernel (chosen conveniently to guarantee the energy inequality \eqref{ineq:energyestimate}) with bandwidth $\veps- 2 \delta$. Notice that in order to get convergence rates for the spectrum it is crucial that the leading term on the right hand side of \eqref{ineq:energyestimate} is $b_\veps(u)$ (with no constants in front). In this construction it is also important to have $\veps > 2 \delta$.  Now, as proved in \cite{trillos2018spectral}, the scaling of $\delta = d_\infty(\mu, \mu_n)$ in terms of $n$ is 
\begin{equation}
\delta  \sim C\begin{cases}  \frac{\log(n)^{3/4}}{n^{1/2}} & m=2 \\  \frac{\log(n)^{1/m}}{n^{1/m}} & m \geq 3,  \end{cases}
\label{eqn:Oldot}
\end{equation}
which means that in dimension $m=2$, $\veps$ is forced to be  larger  than $\frac{\log(n)^{3/4}}{n^{1/2}}$. The extra logarithmic factor with respect to the connectivity threshold has been shown to be unnecessary in order to establish consistency of closely related problems as illustrated in \cite{MShaCCS} and \cite{MullerPenrose}. Here we also remove the extra logarithmic factor, offering also quantitative rates of convergence. For that purpose, we use modified discretization and interpolation maps $\widetilde P$, $\widetilde \I$ constructed as before, but based on a map $\widetilde T$ pushing forward a conveniently chosen measure  $\tilde \mu_n$ with density $\widetilde \rho_n$ (uniformly close to $\rho$) into the empirical measure $\mu_n$. The idea is that $\widetilde{\mu}_n$ can be chosen to be closer to $\mu_n$ in the $\infty$-OT sense than $\mu$ itself. More precisely we have the following proposition which is proved in the Appendix.

\begin{proposition}
	Suppose that $\M, \mu$ satisfy the assumptions from Section \ref{Setup}. Suppose that the quantities $\widetilde \delta , \theta, \veps$ satisfy Assumptions \ref{assumptions}. Then, with probability greater than $1- n \exp( - C n \theta^2 \widetilde{\delta}^m     )$, there exists a probability measure $\widetilde \mu_n$ with density function $\widetilde \rho _n : \M \rightarrow \R$ such that
	\[  \min_{T_{\sharp} \widetilde \mu_n =\mu_n}\sup_{x \in \M} d_\M(x, T(x)) \leq  \widetilde{\delta} ,  \]
	and such that
	\[  \lVert  \rho - \widetilde{\rho}_n \rVert_{L^\infty(\mu)} \leq C \left( \theta   +   \widetilde \delta  \right). \]
	We will use $\widetilde T$ to denote an optimal transport map.
	\label{prop:AuxiliaryDensity}
\end{proposition}

\begin{remark}
	Suppose that in Proposition \ref{prop:AuxiliaryDensity} we take $\theta:=\sqrt{\frac{(\alpha+1) \log(n)}{Cn\widetilde \delta^m} }$  for some $\alpha >1$. Then it  follows that with probability at least $1- \frac{1}{n^\alpha}$ there exists a measure $\widetilde \mu _n$ such that
	\[\min_{T_{\sharp} \widetilde \mu_n =\mu_n}\sup_{x \in \M} d_\M(x, T(x)) \leq  \widetilde{\delta}\]
	and
	\[  \lVert  \rho - \widetilde{\rho}_n \rVert_{L^\infty(\mu)} \leq C \left(    \sqrt{\frac{(\alpha+1) \log(n)}{Cn\widetilde \delta^m} }  +   \widetilde \delta  \right). \]
	In particular, if $\widetilde \delta$ is set to scale like $\frac{\log(n)^{1/m}}{n^{1/m}}\ll \widetilde \delta \ll 1$, we can make $\lVert  \rho - \widetilde{\rho}_n \rVert_{L^\infty(\mu)}$ arbitrarily small.
\end{remark}

%

%
%

After we derive some properties of the new maps $\widetilde P$ and $\widetilde \I$ we will be able to prove the first part of Theorem \ref{thm:evaluesveps}, as well as the first part of Theorem \ref{thm:evectorsveps} following the proof scheme in \cite{trillos2018spectral}. We will also be able to prove the second part of Theorem \ref{thm:evectorsks}. For this we follow some of the steps in the proof of the Davis-Kahan theorem. We will need to use \emph{pointwise consistency} for the graph Laplacian together with our a priori estimates for the error of approximation of eigenvalues, in order to isolate eigenvalues of $\Delta_\rho$ and match them with eigenvalues of $\mathcal{L}^\veps$. We notice that the pointwise consistency for $\eps$-graphs has been established in previous works~\cite{hein2007graph}. For completeness, we give a rather simple proof of consistency in Section \ref{sec:epsLap}. To apply the pointwise consistency we require regularity of the continuum eigenfunctions (see the discussion at the beginning of Section \ref{sec:eigenvectors}), which is an ingredient that has not been utilized in previous works on spectral convergence.

%
%
%

\nc

%
%

Finally, to improve the rates of convergence for eigenvalues (i.e. to obtain the second part of Theorem \ref{thm:evaluesveps}) we will use the error rates for the convergence of eigenvectors as follows. Let us fix $l \in \N$ and let $u$ be a normalized eigenvector of $\L^\veps$ with corresponding eigenvalue $\lambda_l^\veps$. We know a priori that we can find $f \in L^2(\mu)$ a normalized eigenfunction of $\Delta_\rho$ with eigenvalue $\lambda_l$, that is close to $u$, so that in particular
\begin{equation}
\langle u , f \rangle_{L^2(\mu_n)} \geq c >0,  
\label{eqn:apriori}  
\end{equation}
where $c$ is independent of $\veps$, or $n$; in the above formula we interpret $f$ as the restriction of $f: \M \rightarrow \R$ to $X$ (a well defined operation given the continuity of eigenfunctions of $\Delta_\rho$).
We will then arrive to an inequality of the form 
\[  \lvert   \lambda_l^\veps - \sigma_\eta \lambda_l  \rvert  = \frac{\lvert \langle  u, \L^\veps f- \sigma_\eta \Delta_\rho f \rangle  \rvert_{L^2(\mu_n)}}{\lvert \langle u,f \rangle  \rvert_{L^2(\mu_n)}} \leq \frac{1}{c} \lVert(\L^\veps - \sigma_\eta\Delta_\rho) f \rVert_{L^2(\mu_n)}. \]
The above computations show that in order to estimate  $ \lvert   \lambda_l^\veps  - \lambda_l  \rvert$ it suffices to use the pointwise consistency of graph Laplacians applied to the function $f$.

\red

\nc

%
%
%

For the most part, the proof strategy for Theorems \ref{thm:evaluesks} and \ref{thm:evectorsks} is similar to the one described before. However, there are two important modifications we need to make. First, the pointwise estimates for $\L^{k} f (x_i) - \sigma_\eta \Delta^{NN}_\rho f(x_i) $ require new technical computations which are presented in detail in Section \ref{sec:sknn}. These results are new in the literature. Second, regarding the relevant a priori estimates, we must first construct a new interpolation map $\widetilde \I$ as now the connectivity length-scale changes in space. Once the main properties of the new map  $\widetilde I$ have been established our desired results will follow in the exact same way as in the $\veps$-graph case.

We suspect that our eigenvector and eigenvalue rate of $O\left( n^{-1/(m+4)} \right)$ is close to optimal. In the general manifold setting, we suspect $O(n^{-1/m})$ to be the optimal convergence rate for eigenvectors, since this represents the resolution of the point cloud (i.e., the typical inter-point distance). It would be interesting to establish lower bounds for all convergence rates in this paper. We point out that some parts of our proof for the eigenvalue convergence rate can produce lower bounds. For example, in the $\eps$-graph setting, we look forward to  \eqref{eq:keyforlower} in the eigenvalue convergence rate proof, which yields
\[|\lambda^\veps_l - \sigma_\eta \lambda_l|\geq  C_l\sup_{\substack{f\in S(\Delta_\rho,\lambda_l)\\ u\in S(\L^\eps,\lambda^\eps_l)}}|\langle u,\L^\eps f - \sigma_\eta\Delta_\rho f \rangle_{L^2(\mu_n)}|  \]
with probability at least $1-2\exp\left( -cn \right)$, where $S(\Delta_\rho, \lambda)$ is the set of unit norm eigenfunctions of $\Delta_\rho$ with eigenvalue $\lambda$, and  $S(\L^\eps,\lambda)$ is the corresponding set for $\L^\veps$. Obtaining lower bounds on the right hand side would prove lower bounds on the eigenvalue convergence rate. This appears to be a hard problem that will involve tools from the field of stochastic homogenization, since the pointwise consistency errors may average out against $u$ to something smaller than the upper bound given by the application of Cauchy-Schwartz that we made above. We plan to explore lower bounds in a future work.

\section{Pointwise consistency of graph Laplacians}
\label{sec:consistency}

In this section we prove pointwise consistency with a linear rate for our two constructions of the graph Laplacian. The $\eps$-graph Laplacian is considered in Section \ref{sec:epsLap}, while the  undirected $k$-NN Laplacian is considered in Section \ref{sec:sknn}. For the $\eps$-graph Laplacian, the linear rate was established earlier in \cite{hein2007graph}; we give a simpler proof for completeness. Consistency for $k$-NN graph Laplacians was studied in \cite{ting2010analysis}, but the methods used were unable to establish any convergence rates. 

Before giving the pointwise consistency proofs, we review some differential geometry in Section \ref{sec:diffgeom}, and concentration of measure in Section \ref{sec:concentration}.

\subsection{Differential geometry}
\label{sec:diffgeom}

We first briefly review some basic results from differential geometry. We refer the reader to \cite{docarmo1992riemannian} for more details. We write $B_\M (x,r)\subset\M$ to denote the geodesic ball in $\M$ of radius $r$ centered at $x$, while we use $B(x,r)$ to denote Euclidean balls in $\R^m$ or in $\R^d$ depending on context. For each $x\in \M$, $\exp_x:T_x\M \to \M$ is the Riemannian exponential map. Let $K$ be an upper bound on the absolute values of the sectional curvatures, let $R$ be the reach of $\M$, and let $i_0$ be a lower bound on the injectivity radius of $\M$. For any $0 < r < \min\{i_0,K^{-1/2}\}$, $\exp_x:B(0,r)\to \M$ is a diffeomorphism between the ball $B(0,r) \subset T_x\M$ and the geodesic ball $B_\M(x,r)\subset \M$. Let us denote the Jacobian of $\exp_x$ at $v\in B(0,r)\subset T_x\M$ by $J_x(v)$. By the Rauch Comparison Theorem
\begin{equation}\label{eq:dist}
(1+CK|v|^2)^{-1} \leq J_x(v) \leq 1 + CK|v|^2.
\end{equation}
It follows that
\begin{equation}\label{eq:vol}
|\vol_\M(B_\M(x,r))-\alpha_m r^m|\leq CKr^{m+2}.
\end{equation}
We also recall \cite[Proposition 2]{trillos2018spectral} 
\begin{equation}\label{eq:d}
|x-y|\leq d_\M(x,y)\leq |x-y|+ \frac{8}{R^2}|x-y|^3,
\end{equation}
provided $|x-y|\leq R/2$. In the above $d_\M$ denotes the geodesic distance on $\M$.

Throughout this section we always assume $n$ and $\veps$ satisfy Assumptions \ref{assumptions}. We use $C,c>0$ to denote arbitrary constants that depend only on dimension $m$, $\rho$, and on the geometric properties of the manifold $\M$, such as the injectivity radius or sectional curvature. We always have $0 < c < 1$ and $C\geq 1$. 

\subsection{Concentration of measure}
\label{sec:concentration}

We now state a simple concentration inequality, which follows from Bernstein's inequality \cite{boucheron2013concentration}, that is particularly convenient for analysis of graph Laplacians. 
\begin{lemma}\label{lem:ch}
Suppose that $\M, \mu$ satisfy the assumptions from Section \ref{Setup} and let $x_1, \dots, x_n$ be samples from $\mu$. Let $\psi:\M\to \R$ be bounded and Borel measurable. For $x\in \M$ define
\[ \Psi = \sum_{|x_i-x|\leq \eps} \psi(x_i).\]
Then for any $\eps^2 \leq \delta \leq  1$
\begin{equation}\label{eq:con}
\P\left( \left| \Psi - a\right|\geq C\rho_{max}\|\psi\|_\infty \delta n\eps^m\right) \leq 2\exp\left( -c\rho_{max}\delta^2 n\eps^m \right),
\end{equation}
where $\|\psi\|_\infty=\|\psi\|_{L^\infty(B_\M(x,2\eps))}$ and
\begin{equation}\label{eq:amean}
a = n\int_{B_\M(x,\eps)}\psi(y) \rho(y)\,dVol_\M(y).
\end{equation}
\end{lemma}
\begin{remark}
It is important to point out that $a\neq \E[\Psi]$, since the definition of $\Psi$ uses the metric in the ambient space, while in the definition of $a$ we integrate over \emph{geodesic} balls. We also point out that the restriction $\delta\geq \eps^2$ allows us to ignore any effects of the curvature of the manifold.
\end{remark}
\begin{proof}
The argument is similar to  \cite[Lemma 1]{calder2018game}, but we include the proof for completeness.  
Let $Z_i=\one_{\{|x_i-x|\leq \eps\}} \psi(x_i)$. The Bernstein inequality applied to $\sum_{i=1}^n Z_i$ yields
\begin{equation}\label{eq:bernstein}
\P\left( \left|\Psi- \E[\Psi] \right| \geq nt\right) \leq 2\exp\left( - \frac{nt^2}{2(\sigma^2 + \tfrac13 bt)}  \right)
\end{equation}
for any $t>0$, where $\sigma^2 = \text{Var}(Z_i)$, $b>0$ satisfies $|Z_i-\E[Z_i]|\leq b$ almost surely, and 
\[\E[\Psi] = n\int_{B(x,\eps)\cap \M}\psi(y) \rho(y) \, dVol_\M(y).\]
By \eqref{eq:d} we have
\begin{equation}\label{eq:nested}
B_\M(x,\eps) \subset B(x,\eps)\cap \M \subset B_\M(x,\eps + 8\eps^3/R^2) \subset B_\M(x,2\eps)
\end{equation}
provided $8\eps^3/R^2 \leq \eps$, which is guaranteed by the Assumptions \ref{assumptions}. Let us write $\|\psi\|_{\infty}=\|\psi\|_{L^\infty(B_\M(x,2\eps))}$. Then by \eqref{eq:nested} and \eqref{eq:vol} we have
\begin{align}\label{eq:mean}
\left| \E[\Psi] - a\right|&\leq C \rho_{max}\|\psi\|_\infty  n\, Vol_\M(B_\M(x,\eps + 8\eps^3/R^2)\setminus B_\M(x,\eps))\\
&\leq C\rho_{max}\|\psi\|_\infty  n\eps^{m+2},\notag
\end{align}
for $\eps$ sufficiently small. We also have by \eqref{eq:nested} that
\begin{equation}\label{eq:var}
\sigma^2\leq \E[Z_i^2] = \int_{B(x,\eps)\cap \M}\psi(y)^2 \rho(y) \, dVol(y) \leq C\rho_{\max}\|\psi\|^2_\infty \eps^m,
\end{equation}
and
\begin{equation}\label{eq:b}
|Z_i - \E[Z_i]|\leq |Z_i| + |\E[Z_i]| \leq |Z_i| + \E[|Z_i|] \leq 2\|\psi\|_\infty,
\end{equation}
so we can take $b = 2\|\psi\|_\infty$. We now set $t=\rho_{max}\|\psi\|_\infty\eps^m \delta$ in \eqref{eq:bernstein}, for $\delta>0$, and combine this with \eqref{eq:mean} and \eqref{eq:var}  to obtain
\[\P\left( \left|\Psi- \E[\Psi] \right| \geq C\rho_{max}\|\psi\|_\infty (\delta+\eps^2)n \eps^m\right) \leq 2\exp\left( - \frac{c\rho_{max}\delta^2n\eps^m}{1 + \delta}  \right)\]
for constants $C,c>0$ depending only on $\M$. Restriting $\eps^2\leq \delta \leq 1$ completes the proof. 
\end{proof}

\subsection{\texorpdfstring{$\eps$}{epsilon}-graph}
\label{sec:epsLap}

We now turn to proving pointwise consistency of the $\eps$-graph Laplacian. Our main result is the following uniform consistency estimate.
\begin{theorem}[Consistency for $\eps$-graph]\label{thm:con1}
Let $f\in C^3(\M)$.  Then for $\eps \leq \delta \leq \eps^{-1}$ we have
\begin{equation}\label{eq:con1}
\P\left[ \max_{1\leq i\leq n}|\L^\eps f(x_i) - \sigma_\eta \Delta_\rho f(x_i)|\geq C\delta \right]\leq 2n\exp\left(-c\delta^2 n\eps^{m+2}\right),
\end{equation}
where $C$ depends on $\|f\|_{C^3(B_\M(x_i,\eps))}$ and $[f]_{1;B_\M(x_i,2\eps)}$, where $[f]_{1;B_\M(x_i,2\eps)}$ is the Lipschitz constant of the function $f$ when restricted to the ball $B_\M(x_i,2\eps)$.
\end{theorem}

The proof of Theorem \ref{thm:con1} is split into two lemmas, Lemma \ref{lem:cm1} and Lemma \ref{lem:con1}. Lemma \ref{lem:cm1} controls the fluctuations between the graph Laplacian and the nonlocal operator
\begin{equation}\label{eq:nl}
\L^\eps_{nl} f(x) = \frac{1}{\eps^{m+2}}\int_{B_\M(x,\eps)} \eta\left( \frac{|x-y|}{\eps} \right)(f(x)-f(y)) \rho(y)\, dVol_\M(y).
\end{equation}
Lemma \ref{lem:con1} establishes consistency of the nonlocal operator $\L^\eps_{nl}$ to $\Delta_\rho$ as $\eps \to 0$.

\begin{lemma}\label{lem:cm1}
Let $f\in C^1(\M)$.  Then for $x\in \M$ and $\eps \leq \delta \leq \eps^{-1}$
\begin{equation}\label{eq:nlknn}
\P\left[ |\L^\eps f(x) -  \L^\eps_{nl} f(x)|\geq C[f]_{1;B_\M(x,2\eps)}\delta \right]\leq 2\exp\left(-c\delta^2 n\eps^{m+2}\right).
\end{equation}
\end{lemma}
\begin{proof}
 The proof is a direct application of Lemma \ref{lem:ch} with 
\[\psi(y) = \frac{1}{n\eps^{m+2}}\eta\left( \frac{|y-x|}{\eps} \right)(f(x)-f(y)),\]
and $a=\L^\eps_{nl} f(x)$. We simply need to compute
\[\|\psi\|_{L^\infty(B_\M(x,2\eps))} \leq \frac{C[f]_{1;B_\M(x,2\eps)}}{n\eps^{m+1}}.\qedhere\]
\end{proof}

\begin{lemma}\label{lem:con1}
For $f\in C^3(\M)$ and $x\in \M$
\begin{equation}\label{eq:nlcon}
|\L^{\eps}_{nl} f(x) - \sigma_\eta \Delta_\rho f(x)|\leq  C(1+\|f\|_{C^3(B_\M(x,\eps))})\eps.
\end{equation}
\end{lemma}
\begin{proof}
Let us define the intrinsic version of $\L^\eps_{nl}$ to be
\begin{equation}\label{eq:nl2}
\L^{i,\eps}_{nl} f(x) := \frac{1}{\eps^{m+2}}\int_{B_\M(x,\eps)} \eta\left( \frac{d_\M(x,y)}{\eps} \right)(f(x)-f(y)) \rho(y)\, dVol_\M(y).
\end{equation}
Since $\eta$ is Lipschitz, it follows from \eqref{eq:d} that
\begin{equation}\label{eq:error}
|\L^{i,\eps}_{nl} f(x) - \L^{\eps}_{nl} f(x) | \leq C[f]_{1;B_\M(x,\eps)}\eps.
\end{equation}
Let $w(v) = f(\exp_x(v))$ and $p(v) = \rho(\exp_x(v))$; that is $w$ and $p$ are the functions $f$ and $\rho$ expressed in normal Riemannian coordinates. Then we have
\begin{align*}
\L^{i,\eps}_{nl} f(x) &= -\frac{1}{\eps^{m+2}}\int_{B(0,\eps)\subset T_x\M} \eta\left( \frac{|v|}{\eps} \right)(w(v)-w(0)) p(v) J_x(v)\, dv\\
&=- \frac{1}{\eps^{2}}\int_{B(0,1)} \eta\left(|v| \right)(w(\eps v)-w(0)) p(\eps v)J_x(\eps v)\, dv.
\end{align*}
Using the Taylor expansions $J_x(\eps v) = 1+ O(\eps^2)$, $p(\eps v) = p(0) +  \nabla p(0)\cdot v  \eps + O(\eps^2)$, and 
\[w(\eps v) - w(0) = \nabla w(0)\cdot v \eps + \frac{1}{2}v \cdot \nabla^2 w(0) v \eps^2 + O(\|w\|_{C^3(B(0,\eps))}\eps^3),\]
a standard computation yields
\begin{align*}
\L^{i,\eps}_{nl} f(x) = -\sigma_\eta \left( \nabla w(0) \cdot \nabla p(0) + \frac{p(0)}{2}\Delta_\M w(0) \right) + O(c_3\eps)=- \frac{\sigma_\eta}{2p} \div ( p^2 \nabla w )\big\vert_{v=0}+ O(c_3\eps),
\end{align*}
where
\[c_3 = 1+\|w\|_{C^3(B(0,\eps))}.\]
The proof is completed by recalling \eqref{eq:error} and noting that $-\frac{\sigma_\eta}{2p} \div ( p^2 \nabla w )\big\vert_{v=0} = \sigma_\eta \Delta_\rho f(x)$.
\end{proof}

We now give the proof of Theorem \ref{thm:con1}.
\begin{proof}[Proof of Theorem \ref{thm:con1}]
Combining Lemmas \ref{lem:cm1} and \ref{lem:con1} we have 
\begin{equation}\label{eq:partstep}
\P\left[ |\L^\eps f(x) -  \sigma_\eta \Delta_\rho f(x)|\geq C\delta \right]\leq 2\exp\left(-c\delta^2 n\eps^{m+2}\right)
\end{equation}
for any $x\in \M$ and $\eps\leq \delta\leq \eps^{-1}$. Conditioning on $x_i$ and using the law of conditional probability yields
\[\P\left[ |\L^\eps f(x_i) -  \sigma_\eta \Delta_\rho f(x_i)|\geq C\delta \right]\leq 2\exp\left(-c\delta^2 n\eps^{m+2}\right)\]
for any $1\leq i \leq n$. The proof is completed by union bounding over $x_1,\dots,x_n$.
\end{proof}

\subsection{Undirected \texorpdfstring{$k$}-NN graph Laplacian}
\label{sec:sknn}

We now turn to the case of the undirected $k$-NN graph Laplacian $\L^{k}$. The consistency proof here is more involved, since the neighbors in the symmetrized $k$-NN relation do not fall in a ball (even on average), due to the variability of the distribution $\rho$. The additional neighbors added in the symmetrization are in fact important for consistency of the undirected $k$-NN graph Laplacian, and when $\rho$ is not constant the additional points are not symmetrically distributed about $x$. This must be accounted for in the consistency results, and introduces an additional drift term in the limiting Laplace-Beltrami operator. The form of the continuum operator $\Delta_{\rho}^{NN}$ was established in \cite{ting2010analysis}, where the authors proved consistency \emph{without} a convergence rate.

Our main result in this section is the following uniform consistency estimate.
\begin{theorem}[Consistency for undirected $k$-NN Laplacian]\label{thm:sknncon}
Let $f\in C^3(\M)$.  For $1 \leq k \leq cn\eps_\M^m$ and $C(k/n)^{1/m} \leq \delta \leq (k/n)^{-1/m}$ we have
\begin{equation}\label{eq:sknncon}
\P\left( \max_{1 \leq i \leq n}|\L^{k}f(x_i) - \sigma_\eta \Delta^{NN}_\rho f(x_i)| \geq C\delta \right)\leq Cn\exp\left( -c\delta^2(k/n)^{2/m}k\right),
\end{equation}
where $C$ depends on $\|f\|_{C^3(\M)}$.
\end{theorem}

The proof of Theorem \ref{thm:sknncon} follow from Lemmas \ref{lem:cmk} and \ref{lem:conknn} below. Before presenting the lemmas and proofs, we must introduce some notation. We recall $r_k(x,y)=\max\{\eps_k(x),\eps_k(y)\}$, as defined in \eqref{eq:rk}. For $x\in \M$ define $\eps(x)$ by
\begin{equation}\label{eq:epsx}
k = \alpha_m\rho(x) n\eps(x)^m,
\end{equation}
and set
\begin{equation}\label{eq:rxy}
r(x,y) := \max\{\eps(x),\eps(y)\}.
\end{equation}

\begin{lemma}\label{lem:Ne}
	Let $x\in \M$ and suppose that $\eps<\eps_\M$. Then for $\eps^2 \leq \delta \leq 1$
	\begin{equation}\label{eq:Nep}
	\P\left[ \left|N_\eps(x) - \alpha_m \rho(x)n\eps^m\right|\geq C\delta n\eps^m  \right]\leq 2\exp\left( -c\delta^2 n\eps^m \right).
	\end{equation}
\end{lemma}
\begin{proof}
	Applying Lemma \ref{lem:ch} with $\psi\equiv 1$ yields
	\[\P\left[ \left|N_{\eps}(x) - n\int_{B_\M(x,\eps)}\rho(y)\,dVol_\M(y)\right|\geq C\delta n\eps^m  \right]\leq 2\exp\left( -c\delta^2 n\eps^m\right).\]
	By Taylor expansion we have
	\[n\int_{B_\M(x,\eps)}\rho(y)\,dVol_\M(y) = \alpha_m\rho(x)n\eps^m + O(n\eps^{m+2}).\]
	Noting that $\eps^2 \leq \delta$ completes the proof.
\end{proof}

\begin{lemma}\label{lem:epscon}
	For $x\in \M$, $1 \leq k \leq cn\eps_\M^m$ and $C(k/n)^{2/m} \leq \delta \leq 1$ we have
	\begin{equation}\label{eq:epscon}
	\P\left[ |\alpha_m \rho(x)n\eps_k(x)^m - k|\geq C \delta k \right]\leq 4\exp\left( -c\delta^2 k  \right).
	\end{equation}
\end{lemma}
\begin{proof}
	Let $C'$ be the constant from Lemma \ref{lem:Ne} and define $\eps$ by
	\[\alpha_m\rho(x)n\eps^m = k(1+C'\alpha_m^{-1}\rho(x)^{-1}\delta).\]
	Then we have
	\begin{align*}
	\P\left( \alpha_m \rho(x)n\eps_{k}( x )^m>k(1+C\delta) \right)&= \P\left( \eps_k(x) > \eps \right)\\
	&\leq \P\left( N_{\eps}(x) < k \right)\\
	&\leq \P\left( N_{\eps}(x)-\alpha_m\rho(x)n\eps^m < -C'\alpha_m^{-1}\rho(x)^{-1}k\delta  \right)\\
	&\leq \P\left( |N_{\eps}(x)-\alpha_m\rho(x)n\eps^m| \geq C'\delta n\eps^m  \right).
	\end{align*}
	Since $c(k/n)^{1/m} \leq \eps\leq C(k/n)^{1/m}$ the restriction $\eps<\eps_\M$ from Lemma \ref{lem:Ne} is satisfied when $k\leq cn\eps_\M^m$. Therefore, we can invoke Lemma \ref{lem:Ne} to find that
	\[\P\left( n\alpha_m \rho(x)\eps_{k}( x )^m>k(1+C\delta) \right) \leq 2\exp\left( -c\delta^2 k \right).\]
	for $ \eps^2\leq \delta \leq 1$. This establishes one direction of \eqref{eq:epscon}; the proof of the other is similar. 
\end{proof}

We set $\eps_{max}=\max_{x\in \M}\eps(x)$, and define the nonlocal operator
\begin{equation}\label{eq:Lnl}
\L^{k}_{nl}f(x) = \left( \frac{n\alpha_m}{k} \right)^{1+2/m}\int_{B_\M(x,\eps_{max})}\eta\left( \frac{|x-y|}{r(x,y)}\right) (f(x) - f(y))\rho(y)\, d\vol_\M(y).
\end{equation}
Our first lemma shows that the nonlocal operator $\L^{k}_{nl}$ describes the average behavior of the unweighted $k$-nearest neighbor graph Laplacian $\L^{k}$.
\begin{lemma}\label{lem:cmk}
Let $f\in C^1(\M)$.  Then for $x\in \M$ and $C(k/n)^{1/m} \leq \delta\leq (k/n)^{-1/m}$
\begin{equation}\label{eq:glP}
\P\left[ |\L^{k} f(x) -  \L^{k}_{nl} f(x)|\geq C[f]_{1;B_\M(x,2\eps_{max})}\delta \right]\leq C\exp\left( -c\delta^2 (k/n)^{2/m} k \right).
\end{equation}
\end{lemma}
\begin{proof}
Define
\begin{equation}\label{eq:Ak}
A = \left\{ j \, : \, |x-x_j|\leq r_k(x_j,x) \right\},
\end{equation}
and
\begin{equation}\label{eq:A}
A(s) = \left\{ j \, : \,  |x-x_j|\leq r(x_j,x)(1+s) \right\}.
\end{equation}
By Lemma \ref{lem:epscon} we have
\begin{equation}
\P(|\eps_k(x)^m - \eps(x)^m|\geq C\delta \eps(x)^m) \leq 6\exp(-c\delta^2 k)
\label{eqn:auxrelationeps}
\end{equation}
and
\[\P\left[\max_{1\leq j \leq n}|\eps_k(x_j)^m - \eps(x_j)^m|\geq C\delta \eps(x_j)^m\right] \leq 6n\exp(-c\delta^2 k)\]
for $C(k/n)^{2/m} \leq \delta\leq 1$. Fix such a $\delta>0$ and assume that
\begin{equation}\label{eq:eps1}
|\eps_k(x)^m - \eps(x)^m|\leq C\delta \eps(x)^m
\end{equation}
and
\begin{equation}\label{eq:eps2}
\max_{1\leq j \leq n}|\eps_k(x_j)^m - \eps(x_j)^m|\leq C\delta \eps(x_j)^m.
\end{equation}
Then it follows that $A(-C\delta) \subset A\subset A(C\delta)$. We define  
\[Lf(x)= \frac{1}{n}\left( \frac{n\alpha_m}{k} \right)^{1+2/m}\sum_{j\in A(0)} w_{x_jx}^{r(x_j,x)}(f(x) - f(x_j)),\]
and note that
\begin{align*}
|Lf(x) - \L^{k}f(x)|&\leq \frac{C}{n\eps(x)^{m+2}}\Bigg[ \sum_{j\in A(C\delta)\setminus A(-C\delta)}|f(x_j) - f(x)| \\
&\hspace{1in}+ \sum_{j\in A(-C\delta)}|w^{r(x_j,x)}_{x_j,x} - w^{r_k(x_j,x)}_{x_jx}| |f(x_j)-f(x)| \Bigg].
\end{align*}
By the Chernoff bounds we have
\[\P\left(\# A(C\delta) - \# A(-C\delta)\geq C\delta n\eps(x)^m\right) \leq 2\exp(-c\delta^2 k),\] 
for $0 \leq \delta\leq 1$, and by \eqref{eq:eps1} and \eqref{eq:eps} we have
\[\max_{j\in A(-C\delta)} |w_{x_jx}^{r(x_j,x)} - w_{x_jx}^{r_k(x_j,x)}|\leq C\delta.\]
Setting $\delta=(k/n)^{1/m}t$ we have
\begin{equation}\label{eq:LLk}
\P(|Lf(x) - \L^{k}f(x)| \geq C[f]_{1;B_\M(x,2\eps_{max})}t) \leq C\exp(-ct^2(k/n)^{2/m} k)
\end{equation}
for $C(k/n)^{1/m} \leq t\leq (k/n)^{-1/m}$. The proof is completed by invoking Lemma \ref{lem:ch} to obtain 
\[\P(|\L^{k,s}_{nl}f(x) - Lf(x)| \geq C[f]_{1;B_\M(x,2\eps_{max})}t) \leq 2\exp\left( -ct^2 (k/n)^{2/m} k \right)\]
for $C(k/n)^{1/m} \leq t\leq (k/n)^{-1/m}$, and combining with \eqref{eq:LLk}.
\end{proof}

We now establish consistency of the nonlocal operator $\L^{k}_{nl}$ with the weighted Laplacian $\Delta_{\rho}^{NN}$. 
\begin{lemma}\label{lem:conknn}
There exists $c>0$ depending on $\rho$, so that for $f\in C^3(\M)$, $x\in \M$, and $k\leq cn$, we have
\begin{equation}\label{eq:conknn}
|\L^{k}_{nl}f(x) - \sigma_\eta \Delta_{\rho}^{NN} f(x)| \leq C(1 + \|u\|_{C^3(B_\M(x,\eps_{max}))})\eps(x).
\end{equation}
\end{lemma}
\begin{proof}
Throughout the proof we write $\eps=\eps(x)$, and note that 
\[c\left( \frac{k}{n} \right)^{1/m}\leq \eps \leq \eps_{\max}\leq C\left( \frac{k}{n} \right)^{1/m}.\]
We first define the intrinsic version of $\L^{k,s}_{nl}$, given by
\begin{equation}\label{eq:Lnli}
Lf(x) = \left( \frac{n\alpha_m}{k} \right)^{1+2/m}\int_{B_\M(x,\eps_{max})}\eta\left( \frac{d(x,y)}{r(x,y)}\right) (f(x) - f(y))\rho(y)\, dVol_\M(y).
\end{equation}
Since $\eta$ is Lipschitz it follows from \eqref{eq:d} that
\begin{equation}\label{eq:LtoLks}
|Lf(x) - \L^{k}_{nl}f(x)| \leq C[f]_{1;B_\M(x,\eps_{max})}(k/n)^{1/m} \leq C[f]_{1;B_\M(x,\eps_{max})}\eps.
\end{equation}
Let $w(v) = f(\exp_x(v))$ and $p(v) = \rho(\exp_x(v))$. Then we have
\begin{align*}
L f(x) &= -\left( \frac{n\alpha_m}{k} \right)^{1+2/m}\int_{B(0,\eps_{max})\subset T_x\M} \eta\left( \frac{|v|}{r(x,\exp_x(v))} \right)(w(v)-w(0)) p(v) J_x(v)\, dv.
\end{align*}
Let us set $s(v) = (\rho(\exp_x(v))/\rho(x))^{1/m}=(p(v)/p(0))^{1/m}$, so that 
\[r(x,\exp_x(v)) = \eps\,\max\{1,s(v)^{-1}\} = \frac{\eps(x)}{\min\{1,s(v)\}}.\]
Then making a change of variables $v'=v/\eps$, and renaming $v'$ as  $v$,  yields
\begin{equation}\label{eq:Lu}
L f(x) =- \left( \frac{n\alpha_m}{k} \right)^{2/m}\int_{B(0,C)} \eta\left(|v|\min\{1,s(\eps v)\} \right)(w(\eps v)-w(0)) \frac{p(\eps v)}{p(0)}J_x(\eps v)\, dv,
\end{equation}
where $C=\eps^{-1}\eps_{max}$.
We now use the Taylor expansions $J_x(\eps v) = 1+ O(\eps^2)$, $p(\eps v) = p(0) +  \nabla p(0)\cdot v  \eps + O(\eps^2)$, 
\[w(\eps v) - w(0) = \nabla w(0)\cdot v \eps + \frac{1}{2}v \cdot \nabla^2 w(0) v \eps^2 + O(\|w\|_{C^3(B(0,\eps_{max}))}\eps^3),\]
and
\[s(\eps v) = 1 + \frac{1}{m}\nabla \log p(0)\cdot v\eps + O(\eps^2),\]
to obtain
\begin{align*}
Lf(x) &=- \left( \frac{n\alpha_m}{k} \right)^{2/m}\int_{B} \eta\left(|v|(1+\tfrac{\eps}{m}[\nabla \log p(0)\cdot v]_- ) \right)(\nabla w(0)\cdot v \eps + \frac{1}{2}v \cdot \nabla^2 w(0) v \eps^2)\\
&\hspace{3.5in}(1 +  \nabla \log p(0)\cdot v  \eps)\, dv + O(\eps),
\end{align*}
where $a_-=\min\{0,a\}$ and 
\[B = \{v\in \R^m\, : \, |v|(1+\tfrac{\eps}{m}[\nabla \log p(0)\cdot v]_-) \leq 1\}.\]
We now make the change of variables
\[z = \Phi(v) := v(1+\tfrac{\eps}{m}[\nabla \log p(0)\cdot v]_- ).\]
For sufficiently small $\eps>0$, $\Phi$ is invertible and
\[v = \Phi^{-1}(z)= z(1-\tfrac{\eps}{m}[\nabla \log p(0)\cdot z]_-  + O(\eps^2)).\]
Note that for $v$ with $\nabla \log p(0)\cdot v>0$, we have $D\Phi(v)=I$ and $\det(D\Phi(v))=1$. For $v$ with $\nabla \log p(0)\cdot v<0$ we have
\[D\Phi(v) = I + \frac{\eps}{m}\left( (\nabla  \log p(0)\cdot v)I + \nabla \log p(0) \otimes v \right).\]
Using the Taylor expansion $\text{det}(I+\eps X) = 1 + \eps \text{Tr}(X) + O(\eps^2)$ we have
\begin{align*}
\det(D\Phi(v)) &= 1 + \frac{\eps}{m}\left((\nabla  \log p(0)\cdot v)m + (\nabla  \log p(0)\cdot v)  \right) +O(\eps^2)\\
&= 1 + \eps(1+\tfrac{1}{m})(\nabla  \log p(0)\cdot v) + O(\eps^2). 
\end{align*}
Thus, for all $v$ we have
\[|\det(D\Phi(v))|^{-1} = 1 - \eps(1+\tfrac{1}{m})[\nabla  \log p(0)\cdot v]_- + O(\eps^2),\]
and so
\[dv = (1 - \eps(1+\tfrac{1}{m})[\nabla  \log p(0)\cdot v]_-+ O(\eps^2))dz.\] 
Therefore
\begin{align*}
&L f(x) \\
&=-\left( \frac{n\alpha_m}{k} \right)^{2/m}\int_{B(0,1)} \eta\left(|z| \right)\left(\nabla w(0)\cdot z \eps - \frac{\eps^2}{m}[\nabla \log p(0)\cdot z]_-\nabla w(0)\cdot z  + \frac{1}{2}x \cdot \nabla^2 w(0) x \eps^2\right)\\
&\hspace{2in}(1 + \eps\nabla \log p(0)\cdot z)\left(1 -\eps(1+\tfrac{1}{m})[\nabla \log p(0)\cdot z]_-\right)\, dz + O(\eps)\\
&=-\left( \frac{n\alpha_m}{k} \right)^{2/m}\int_{B(0,1)} \eta\left(|z| \right)\left(\nabla w(0)\cdot z \eps - \frac{\eps^2}{m}[\nabla \log p(0)\cdot z]_-\nabla w(0)\cdot z  + \frac{1}{2}x \cdot \nabla^2 w(0) x \eps^2\right)\\
&\hspace{2in}\left(1 -\eps(1+\tfrac{1}{m})[\nabla \log p(0)\cdot z]_- + \eps\nabla \log p(0)\cdot z\right)\, dz + O(\eps)\\
&=-p(0)^{-2/m}\int_{B(0,1)} \eta\left(|z| \right)\Big( (\nabla \log p(0)\cdot z)(\nabla w(0)\cdot z) + \frac{1}{2}x \cdot \nabla^2 w(0) x \\
&\hspace{2.5in} - (1+\tfrac{2}{m})[\nabla \log p(0)\cdot z]_-(\nabla w(0)\cdot z) \Big)\, dz + O(\eps)\\
&=-\sigma_\eta p(0)^{-2/m}(\nabla \log p(0)\cdot \nabla w(0)+\tfrac{1}{2}\Delta w(0)) \\
&\hspace{1in}- p(0)^{-2/m}( 1+\tfrac{2}{m})\underbrace{\int_{B(0,1)} \eta\left(|z| \right)[\nabla \log p(0)\cdot z]_-(\nabla w(0)\cdot z)\, dz}_I + O(\eps).
\end{align*}
In the final integral above, let $A$ be an orthogonal transformation so that
\[A\nabla \log p(0) = |\nabla \log p(0)|e_m\]
and make the change of variables $y = Az$ to deduce
\begin{align*}
I&=|\nabla \log p(0)|\int_{B(0,1)} \eta\left(|y| \right)\min\{y_m,0\}(A\nabla w(0)\cdot y)\, dy \\
&=|\nabla \log p(0)|\sum_{i=1}^m [A\nabla w(0)]_i \int_{B(0,1)} \eta\left(|y| \right)\min\{y_m,0\}y_i\, dy \\
&=|\nabla \log p(0)| [A\nabla w(0)]_m \int_{B(0,1)} \eta\left(|y| \right)\min\{y_m,0\}y_m\, dy \\
&=\frac{1}{2}|\nabla \log p(0)| [A\nabla w(0)]_m \int_{B(0,1)} \eta\left(|y| \right)y_m^2\, dy \\
&=\frac{\sigma_\eta}{2}|\nabla \log p(0)| [A\nabla w(0)]_m\\
&=\frac{\sigma_\eta}{2}\nabla \log p(0)\cdot \nabla w(0).
\end{align*}
This gives
\begin{align*}
L f(x) &=- \frac{\sigma_\eta}{2}p(0)^{-2/m}(\Delta w(0) + (1-\tfrac{2}{m}) \nabla \log p(0) \cdot \nabla w(0)) + O(\eps)\\
&= -\frac{\sigma_\eta}{2}\frac{1}{p(0)}\div(p^{1-2/m}\nabla w) + O(\eps) = \sigma_\eta \Delta_{\rho}^{NN} f + O(\eps),
\end{align*}
which completes the proof.
\end{proof}

\section{Proofs of main results}
\label{sec:mainproofs}

Here we prove all of our main results. The structure of the proofs is exactly the same for the $\veps$-graph and the $k$-NN settings.

\subsection{\texorpdfstring{$\veps$}{epsilon}-graph.}
\label{sec:proofsveps}

\nc

Let $\veps, \widetilde \delta, \theta$ be positive numbers satisfying Assumptions \ref{assumptions}. Associated to these numbers we consider the density $\widetilde{\rho}_n$ from Proposition \ref{prop:AuxiliaryDensity} (which exists with probability greater than $1-n\exp( -Cn\theta^2\tilde{\delta}^2 )$) and we let $\widetilde{T} $ be an $\infty$-OT map between $\widetilde{\mu}_n$ and $\mu_n$. Let $\widetilde U_1 , \dots, \widetilde U_n$ be defined by
\[  \widetilde U _i := \widetilde T_n ^{-1}(\{ x_i \}).   \]
We can then define the \emph{contractive discretization} map $\widetilde P\colon L^2( \mu) \to L^2(\mu_n)$ by 
\begin{equation}
(\widetilde Pf)(x_i) \coloneqq n \cdot \int_{\widetilde U_i} f(x) \widetilde p_n(x) dx, \quad f \in  L^2(\mu ),
\label{def:P1}
\end{equation}
and the extension map $\widetilde P^*\colon L^2(\mu_n) \to L^2( \widetilde \mu_n)$ by
\begin{equation}
(\widetilde P^*u)(x) \coloneqq \sum_{i=1}^n u(x_i) \mathds{1}_{\widetilde U_i} (x), \quad u \in L^2(\mu_n).
\label{def:P*1}
\end{equation}
We note that $\widetilde P^*u$ can be written as $\widetilde P^*u = u \circ \widetilde T$. We then define the \textit{interpolation} map $\widetilde \I\colon L^2(\mu_n) \to \Lip(\M)$
\begin{equation}
\widetilde \I u \coloneqq \Lambda_{\veps - 2 \widetilde \delta} P^*u
\label{eqn:InterpolatingOp}
\end{equation}
where $\Lambda_{\veps - 2 \widetilde \delta}$ is a convolution operator using the kernel $K_r(\cdot, \cdot)$ (defined below) with bandwidth $\veps - 2 \widetilde \delta$. To define the kernel $K_r$ we let $\psi \colon [0,\infty) \to [0,\infty)$ be given by
\begin{equation}
  \psi(t) \coloneqq \frac{1}{\sigma_\eta} \int_{t}^{\infty} \eta(s)s  ds,
\label{eqn:psi}
\end{equation}
and set
\[  K_r(x,y) := \frac{1}{r^m} \psi \left( \frac{d_\M(x,y)}{r} \right), \]
The operator $\Lambda_r$ then takes the form
\[   \Lambda_r f(x) := \frac{1}{\tau(x)} \int_{\M} K_{r}(x,y) f(y) d\mu(y),  \]
where in the above $\tau(x)$ is a normalization factor given by
\[  \tau(x) := \int_{\M}  K_{r}(x,y) d\mu(y), \]
and serves as normalization constant. Note that in the above we integrate with respect to the density $\rho$ and not with respect to $\widetilde \rho_n$. This is because $\widetilde \rho_n$ is discontinuous, and for some of the estimates that we will use later on, we need to integrate with respect to a smoother density. 

In order to prove the first part of Theorem \ref{thm:evaluesveps} we use the following two propositions. These two results contain the fundamental a priori estimates that we use in the sequel.

\begin{proposition}[Inequality for Dirichlet energies]
	\label{prop:localnonlocal}
	Let $\veps$, $\widetilde{\delta}$, and $\theta$ be fixed but small enough numbers satisfying Assumptions \ref{assumptions}. Then, with probability greater than $1-  Cn\exp( -Cn \theta^2 \widetilde{\delta}^m )$ we have: \nc
\begin{enumerate}
	\item For any $f \in L^2(\mu)$,
	\[  b_\veps(\widetilde P f ) \leq \left(1+  C\left(\frac{\widetilde \delta}{\veps} + \veps  + \theta \right) \right) \sigma_\eta  D_2(f),\] 
	\item For any $u \in L^2(\mu_n)$,
	 \[  \sigma_\eta D_2(\widetilde {\mathcal{I}} u) \leq \left(1+ C\left(  \frac{\widetilde \delta}{\veps} +  \veps + \theta\right) \right)  b_\veps(u)    ,\]
\end{enumerate}

%
%
	\label{prop:LocalNonLocal}
\end{proposition}

\begin{proposition}[Discretization and interpolation maps are almost isometries]
Let $\veps$, $\widetilde{\delta}$, and $\theta$ be fixed but small enough numbers satisfying Assumptions \ref{assumptions}. Then, with probability at least $1- C n \exp (-Cn \theta^2 \widetilde{\delta}^m )$ we have:
	\begin{enumerate}
		\item For every $f \in L^2(\mu)$,
		\[  \left  \lvert   \lVert f  \rVert_{L^2(\mu)} ^2  -  \lVert   \widetilde{P} f  \rVert_{L^2(\mu_n)}^2  \right  \rvert \leq   C \widetilde \delta \lVert f \rVert_{L^2(\mu)} \sqrt{ D(f)}  + C(\theta + \widetilde \delta) \lVert f \rVert^2_{L^2(\mu)}.  \]
		\item For every $u \in L^2(\mu_n)$,
		\[  \left  \lvert  \lVert  u  \rVert_{L^2(\mu_n)}^2   -  \lVert   \widetilde{\mathcal {I}} u  \rVert_{L^2(\mu)}^2   \right \rvert \leq   C \veps \lVert u  \rVert_{L^2(\mu_n)} \sqrt{b_\veps(u)} + C(\theta + \widetilde \delta) \lVert u \rVert^2_{L^2(\mu_n)}.   \]	
	\end{enumerate}
	\label{prop:almostisometries}
\end{proposition}

Before proving these two propositions it will be convenient to introduce two intermediate (non-local) Dirichlet energies of interest and establish a connection between them. First, we define the \textit{non-local} energy
\begin{equation}
\label{eqn:ErfV1}
\widetilde E_r(f):=\int_\M \int_\M \frac{1}{r^{m+2}} \eta \left( \frac{d_\M(x,y)}{r} \right)(f(x) - f(y))^2  \widetilde \rho(x) \widetilde \rho_n(y) d\vol_\M(x) d\vol_\M(y), \quad f \in L^2(\M, \mu),
\end{equation} 
and the closely related
\begin{equation}
\label{eqn:ErfV}
 E_r(f):=\int_\M \int_\M \frac{1}{r^{m+2}} \eta \left( \frac{d_\M(x,y)}{r} \right)(f(x) - f(y))^2  \rho(x) \rho(y) d\vol_\M(x) d\vol_\M(y), \quad f \in L^2(\M, \mu).
\end{equation} 
Notice that the only difference between $\widetilde E _r$ and $E_r$ is the density we integrate with respect to. Moreover, for the density $\widetilde \rho_n$ from Proposition \ref{prop:AuxiliaryDensity} we have 
\begin{equation}
  (1- C(\theta + \widetilde \delta)) E_r(f)    \leq  \widetilde E_r(f) \leq (1+  C(\theta + \widetilde {\delta}) ) E_r (f), \quad \forall f \in L^2(\mu), 
\label{lemma:NonlocalIneqs}
\end{equation}
for some constant $C$, which follows from the fact that
\[  \rho(x) \leq \widetilde \rho_n (x) + \lVert \rho - \widetilde \rho_n  \rVert_{L^\infty(\mu)} \leq \widetilde{\rho}_{n}(x)  + \frac{C(\theta + \widetilde \delta)}{\rho_{min}} \widetilde \rho_n(x)  \]
and
\[ \widetilde \rho_n (x) \leq \rho (x) + \lVert \rho - \widetilde \rho_n  \rVert_{L^\infty(\mu)} \leq \rho(x) +   \frac{C(\theta + \widetilde \delta)}{\rho_{min}}  \rho(x).\]

We are now ready to prove Proposition \ref{prop:LocalNonLocal} and Proposition \ref{prop:almostisometries}.

\begin{proof}[Proof of Proposition \ref{prop:LocalNonLocal}]
First, we can use Lemma 5 in \cite{trillos2018spectral} to conclude that for all $r>0$ small enough and all $f \in L^2(\mu)$, 
\[  E_r(f) \leq  (1+ C_{\rho} r  + Cm K  r^2 ) \sigma_\eta  D_2(f),\]
and Lemma 9 in \cite{trillos2018spectral} to get
\[ \sigma_\eta D_2(\Lambda_r f) \leq (1+ C r + Cm K r^2 )\cdot (1+ C(1+ 1/\sigma_\eta)mK r^2  )    E_{r}(f). \]
We can then use \ref{prop:AuxiliaryDensity} to obtain
\begin{equation} 
 \widetilde E_r(f) \leq (1+ C r + Cm K  r^2 )(1+  C(\theta + \widetilde \delta ) ) \sigma_\eta  D_2(f),
 \label{aux:eqn51}
 \end{equation}
as well as
\begin{equation}
 \sigma_\eta D_2(\Lambda_r f) \leq (1+ C r + Cm K r^2 )\cdot (1+ C(1+ 1/\sigma_\eta)mK r^2  ) (1+ C(\theta + \widetilde \delta)) \widetilde E_{r}(f). 
 \label{aux:eqn52}
 \end{equation}

 On the other hand, if we replace $P, P^*, \delta, T_n$ and $\rho$ with $\widetilde P , \widetilde P^* , \widetilde \delta , \widetilde T _n$ and $\widetilde \rho _n$,  (where recall $\widetilde \delta$ is an upper bound for the $\infty$-OT distance between $\widetilde \mu_n$ and $\mu_n$), we can copy word for word the proofs in Lemmas 13 and 14 in \cite{trillos2018spectral} to deduce
 \[  b_\veps(\widetilde P  f) \leq \left(1+ C\frac{\widetilde \delta }{\veps} \right)  \widetilde E_{\veps + 2\widetilde \delta}(f)\] 
 as well as
 \begin{equation}
 \widetilde E_{\veps - 2 \widetilde \delta}(\widetilde P^* u)   \leq \left(1+ C\frac{\widetilde \delta}{\veps}\right) b_\veps(u)   
\label{NonLocal-Discrete}
\end{equation}
for every $f \in L^2(\mu)$ and $u \in L^2(\mu_n)$. Putting together the above estimates with \eqref{aux:eqn51} and \eqref{aux:eqn52} with $r=\veps - 2 \widetilde \delta$, and using the smallness Assumptions \eqref{assumptions} on $\veps , \widetilde \delta$ and $\theta$, we obtain the desired inequalities. 
\end{proof}

\begin{proof}[Proof of Proposition \ref{prop:almostisometries}]
	
	The proof follows the same ideas used when proving the analogous results in \cite{trillos2018spectral}. The setting is slightly different because the discretization and interpolation maps have changed.

	To prove the first assertion, we start by noticing that
	\[ \lVert   f \rVert_{L^2(\mu)}^2 = \int_{\M} \left(  f (x) \right)^2 \frac{\rho(x) - \widetilde \rho _n(x)}{\rho(x)} \rho(x) d\vol_\M(x)   +  \lVert  f \rVert_{L^2(\widetilde \mu_n)}^2.   \]
	The first term on the right hand side is in absolute value less than 
	\[  \frac{\lVert \rho - \widetilde \rho_n \rVert_{L^\infty (\mu)}}{\rho_{min}} \lVert   f \lVert^2_{L^2(\mu)}, \]
	and so
	\begin{equation}
	\left \lvert \lVert   f \rVert_{L^2(\widetilde \mu_n)}^2 - \lVert   f \rVert_{L^2(\mu)}^2  \right \rvert \leq C(\theta + \widetilde \delta) \lVert f \rVert^2_{L^2(\mu)}.    
	\label{ineq:intrhontilderhon}
	\end{equation}
	Notice also that by definition of $\widetilde P^* $ we have
	\[ \lVert  \widetilde P^* \widetilde P f \rVert^2_{L^2(\widetilde \mu_n)} = \lVert  \widetilde P f \rVert^2_{L^2(\mu_n)}. \]
	It follows that,
	\begin{align*}
	\left \lvert  \lVert  \widetilde P f \rVert_{L^2(\mu_n)}^2  - \lVert f  \rVert^2_{L^2(\mu)} \right \rvert  & \leq    \left \lvert  \lVert  \widetilde P^* \widetilde P f \rVert_{L^2(\widetilde \mu_n)}^2  - \lVert f  \rVert^2_{L^2(\widetilde \mu_n)} \right \rvert +   \left \lvert  \lVert f  \rVert^2_{L^2(\widetilde \mu_n)}  - \lVert f  \rVert^2_{L^2(\mu)}  \right \rvert 
	\\& \leq C \lVert f \rVert_{L^2(\mu)} \left \lvert \lVert \widetilde P ^* \widetilde P f     \rVert_{L^2(\widetilde \mu_n)}  - \lVert f \rVert_{L^2(\widetilde \mu_n)}  \right \rvert + C(\theta + \widetilde \delta)\lVert f \rVert^2_{L^2(\mu)}
	\\& \leq  C \lVert f \rVert_{L^2(\mu)}  \lVert \widetilde P ^* \widetilde P f  - f  \rVert_{L^2(\widetilde \mu_n)}  + C(\theta + \widetilde \delta)\lVert f \rVert^2_{L^2(\mu)}.
	\end{align*}
	On the other hand,
	\begin{align*}
	\begin{split}
	\lVert \widetilde P ^* \widetilde P f  - f \rVert_{L^2(\widetilde \mu_n)}^2 &= \sum_{i=1}^n \int_{\widetilde U_i} \left( \widetilde P ^* \widetilde P f (x) - f(x)   \right)^2 \widetilde \rho_n(x) d\vol_\M(x)
	\\&= \sum_{i=1}^n \int_{\widetilde U_i} \left( \widetilde P f (x_i) - f(x)   \right)^2 \widetilde \rho_n(x) d\vol_\M(x)  
	\\&= \sum_{i=1}^n \int_{\widetilde U_i} \left( n \int_{\widetilde U_i} f(y) \widetilde \rho_n(y) d\vol_\M(y)   - f(x)   \right)^2 \widetilde \rho_n(x) d\vol_\M(x)+
	\\&=  \sum_{i=1}^n \int_{\widetilde U_i} \left( n \int_{\widetilde U_i} (f(y)-f(x)) \widetilde \rho_n(y) d\vol_\M(y) \right)^2 \widetilde \rho_n(x) d\vol_\M(x)
	\\& \leq n\sum_{i=1}^n \int_{\widetilde U _i} \int_{\widetilde U _i} (f(y) - f(x))^2 \widetilde \rho_n(x) \widetilde \rho_n(y) d\vol_\M(y) d\vol_\M(x).
	\end{split}
	\end{align*}
	Let us now show that the last term in the above chain of inequalities can be controlled by $C \widetilde \delta^2 E_{2 \widetilde \delta } (f)$, where actually $C$ is a constant that only depends on dimension. For this purpose we use an idea described in Lemma 3.4 in \cite{BIK} to estimate each of the terms
	\[ n\int_{\widetilde U _i} \int_{\widetilde U _i} (f(y) - f(x))^2 \widetilde \rho_n(x) \widetilde \rho_n(y) d\vol_\M(y) d\vol_\M(x).\]

	For fixed $x, y \in \widetilde U_i$ let $W := \M \cap B(x , 2\widetilde \delta) \cap B(y, 2\widetilde \delta) $, where $B(x,2\widetilde \delta)$ is the Euclidean ball of radius $2\widetilde \delta$ centered at $x$. For every  $z \in W$ we have
	\[ |f(x) - f(y)|^2 \leq 2 |f(x) - f(z)|^2 + 2 |f(y)- f(z)|^2,  \] 
	and in particular 
	\begin{align*}
	|f(x) - f(y)|^2 & \leq 2 \frac{1}{\vol_\M(W)} \int_{W}|f(x) - f(z)|^2 d\vol_\M(z) + 2 \frac{1}{\vol_\M(W)} \int_{W}|f(y) - f(z)|^2 d\vol_\M(z)
	\\& \leq  \frac{2}{\vol_\M(W)} \int_{\M} \eta \left(\frac{x - z}{2\widetilde \delta} \right)|f(x) - f(z)|^2 d\vol_\M(z)  
	\\&+ \frac{2}{\vol_\M(W)} \int_{\M} \eta \left(\frac{y - z}{2\widetilde \delta} \right)|f(y) - f(z)|^2 d\vol_\M(z)
	\\& \leq  \frac{C}{\widetilde \delta^m} \int_{\M} \eta \left(\frac{x - z}{2\widetilde \delta} \right)|f(x) - f(z)|^2 d\vol_\M(z)  + \frac{C}{\widetilde \delta^m} \int_{\M} \eta \left(\frac{y - z}{2\widetilde \delta} \right)|f(y) - f(z)|^2 d\vol_\M(z).
	\end{align*}
 Integrating with respect to $x$ and $y$ in both sides of the inequality we get
 \begin{align*}
&	n  \int_{\widetilde U _i} \int_{\widetilde U_i }  |f(x) - f(y)|^2 \widetilde \rho_n(x) \widetilde \rho_n(y) d\vol_\M(x) d\vol_\M(y) \\ &\leq    \frac{2C}{\widetilde \delta^m} \int_{\widetilde U_i} \int_\M \eta \left( \frac{|x-z|}{2 \widetilde \delta} \right) | f(x) - f(z)|^2 \widetilde \rho_n(x) \widetilde \rho_n(z) d\vol_\M(x) d\vol_\M(z).	   
\end{align*}
	Summing over all $i=1, \dots, n$ we deduce from the above expression that
	\[ \lVert \widetilde P ^* \widetilde P f  - f \rVert_{L^2(\widetilde \mu_n)}^2 \leq C \widetilde \delta^2 \widetilde E_{2\widetilde \delta }(f) \leq  C \widetilde \delta^2 D_2(f)  ,  \]
	where for the last inequality we have used the same arguments as in the proof of Proposition \ref{prop:LocalNonLocal}. 
	
	To show the second assertion we notice that from $\lVert  u \rVert_{L^2(\mu_n)}=  \lVert \widetilde P^* u \rVert_{L^2(\widetilde \mu_n)}$ and the triangle inequality we get
	\begin{align*}
	\begin{split}
	\left \lvert \lVert \widetilde \I u \rVert _{L^2(\widetilde \mu _n)}   -  \lVert u \rVert _{L^2(\mu_n)}   \right \rvert & \leq   \lVert \Lambda_{\veps - 2\widetilde \delta} \widetilde P^* u  - \widetilde P^* u    \rVert _{L^2(\widetilde \mu _n)} 
	\\&  \leq \left( 1+ \frac{\lVert \rho - \widetilde \rho_n \rVert_\infty}{\rho_{min}} \right) \cdot \lVert \Lambda_{\veps - 2\widetilde \delta} \widetilde P^* u  - \widetilde P^* u    \rVert _{L^2(\mu)} 
	\\&  \leq  \left( 1+ \frac{\lVert \rho - \widetilde \rho_n \rVert_\infty}{\rho_{min}} \right) \cdot C \veps  \sqrt{ E_{\veps - 2 \widetilde \delta}( \widetilde P^* u) }
	\\& \leq  C  \veps \sqrt{b_\veps(u)},
	\end{split}
	\end{align*}
	where for the third inequality we have used Lemma 8 in \cite{trillos2018spectral} (notice that our definition of $E_r$ has an extra factor of $r^{-2}$ when compared to the definition in \cite{trillos2018spectral}), and for the last one we have used Lemma \ref{lemma:NonlocalIneqs} and \eqref{NonLocal-Discrete}. Also, notice that
	\[   \lVert \widetilde I   u  \rVert_{L^2(\widetilde \mu_n)} = \lVert \Lambda_{r} \widetilde P^* u  \rVert_{L^2(\widetilde \mu_n)} \leq  C \lVert \widetilde P^* u \rVert_{L^2(\widetilde \mu_n)} = C \lVert u \rVert_{L^2(\mu_n)},   \]
	where the inequality follows from Lemma 8 in \cite{trillos2018spectral} and the fact that by introducing a multiplicative constant we can change integrals with respect to $\rho$ with integrals with respect to $\widetilde \rho_n$ and vice-versa (according to Proposition \eqref{prop:AuxiliaryDensity}). Therefore, 
	\begin{align*}
	\begin{split}
	\left \lvert \lVert \widetilde \I u \rVert^2 _{L^2(\widetilde \mu _n)}   -  \lVert u \rVert^2_{L^2(\mu_n)} \right \rvert & \leq \left \lvert \lVert \widetilde \I u \rVert _{L^2(\widetilde \mu _n)}   -  \lVert u \rVert _{L^2(\mu_n)} \right \rvert  \left( \lVert \widetilde \I u \rVert _{L^2(\widetilde \mu _n)}   +  \lVert u \rVert _{L^2(\mu_n)}   \right)
	\\& \leq \left \lvert \lVert \widetilde \I u \rVert _{L^2(\widetilde \mu _n)}   -  \lVert u \rVert _{L^2(\mu_n)} \right \rvert  \left( \lVert  \widetilde \I u \rVert _{L^2(\widetilde \mu _n)}   +  \lVert u \rVert _{L^2(\mu_n)}   \right)
	\\& \leq C \veps \sqrt{b_\veps(u)}  \lVert u \rVert_{L^2(\mu_n)}.
	\end{split}
	\end{align*}
	Finally, we use \eqref{ineq:intrhontilderhon} to compare $\lVert \widetilde \I u \rVert^2 _{L^2(\widetilde \mu _n)} $ and $\lVert \widetilde \I u \rVert^2 _{L^2( \mu)} $.

\end{proof}

With Proposition \ref{prop:LocalNonLocal} and \ref{prop:almostisometries} in hand the proof of the first part of Theorem \ref{thm:evaluesveps} now follows from standard arguments.

\begin{proof}[Proof of 1) in Theorem \ref{thm:evaluesveps}]
Let $f_1, \dots, f_l$ be an orthonormal set (in $L^2(\mu)$) of eigenfunctions of $\Delta_\rho$ corresponding to its first $l$ eigenvalues. For $i=1, \dots, l$ let
\[  v_i :=  \widetilde P  f_i.\]
Applying the first part of Proposition \ref{prop:almostisometries} to functions $f$ of the form:
\[ f:= f_i - f_j \] 
we can get the bound, 
\[  \lvert \langle f_i, f_j\rangle_{L^2(\mu)} - \langle v_i, v_j\rangle_{L^2(\mu_n)} \rvert   \leq C \sqrt{\lambda_l}  \widetilde \delta  + C(\theta + \widetilde \delta )  \leq \frac{1}{2l},  \]
where the last inequality is valid thanks to our smallness assumption for $\widetilde \delta$ and $\theta$. Since $\langle  f_i , f_j \rangle_{L^2(\mu)}= \delta_{ij}$, the above inequality implies that the vectors $v_1, \dots, v_l $ are linearly independent, and so the subspace $S:=  \text{Span}\{ v_1, \dots
, v_l \}$ has dimension $l$. We can use \eqref{eqn:varcharacveps} to conclude that 
\[  \lambda_l^\veps \leq  \frac{1}{2} \max_{\substack{v \in  S \\ \lVert v\rVert_{L^2(\mu_n)}=1}}   b_\veps(v). \]
Now, each element $v$ of $S$ is of the form 
\[ v = \sum_{i=1}^l a_i \widetilde P f_i,   \]
for some coefficients $a_i$, so that 
\[ v = \widetilde P( \sum_{i=1}^l a_i  f_i) =: \widetilde P(f), \]
and in particular,
\[ \frac{1}{2}D_2(f) = \langle  \Delta_{\rho} f , f \rangle_{L^2(\mu)} \leq \lambda_l \lVert f \rVert^2_{L^2(\mu)}. \]
By Proposition \ref{prop:localnonlocal} we know that,
\[\frac{1}{2} b_\veps(v) = \frac{1}{2} b_\veps(\widetilde P f ) \leq (1+  C\frac{\widetilde \delta}{\veps} + C\veps  + C\theta + C\widetilde \delta) \frac{\sigma_\eta}{2}  D(f) \leq  (1+ C\frac{\widetilde \delta}{\veps} + C\veps  + C\theta + C\widetilde \delta)  \sigma_\eta \lambda_l \lVert f \rVert_{L^2(\mu)}^2.\]
If $\lVert v\rVert_{L^2(\mu_n)}=1$, Proposition \ref{prop:almostisometries} implies 
\[ \lVert f \rVert_{L^2(\mu)}^2 \leq 1+  C\widetilde \delta \sqrt{\lambda_l} + C (\theta +\widetilde \delta).  \] 
Thus, 
\[  \lambda_l ^\veps \leq \sigma_\eta \lambda_l  +   C(\widetilde \delta \sqrt{\lambda_l} +  \frac{\widetilde \delta}{\veps} +\veps   +\theta    ) \lambda_l.\]
This establishes the upper bound for $\lambda_l^\veps$ in terms of $\lambda_l$.

To obtain the lower bound, we let $u_1, \dots, u_l$ be an orthonormal set of eigenvectors of $\L^\veps$ corresponding to its first $l$ eigenvalues. We define functions $f_i \in L^2(\mu)$, according to
\[ f_i := \widetilde{\mathcal{I}} u_i , \quad i=1, \dots, l. \]
Thanks to the smallness assumption of the parameters $\widetilde \delta, \veps, \theta$, we see that $S:= \Span \{ f_1, \dots, f_l \}$ has dimension $l$. Just as in the proof of the upper bound we can use the second parts of Propositions \ref{prop:localnonlocal} and \ref{prop:almostisometries} to conclude that
\[ \sigma_\eta \lambda_l \leq  \lambda_l^\veps + C\left( \veps \sqrt{\lambda_l^\veps} + \frac{\widetilde \delta}{\veps}  + \veps + \theta  \right) \lambda_l^\veps.  \]
However, using the upper bound for $\lambda_l^\veps$, we can replace  all the appearances of $\lambda_l^\veps$ in the above error terms with $\lambda_l$, and obtain the desired inequality. 
%
%

\end{proof}

\nc

We can now prove Theorem \ref{thm:evectorsveps}. We start with its second part.

\begin{proof}[Proof of 2) in Theorem \ref{thm:evectorsveps}]
\medskip

A given eigenvalue $\lambda>0$ of $\Delta_{\rho}$  is equal to  $\lambda_{i+1}, \dots, \lambda_{i+k}$ for some $i$ and some $k$, where $k$ is the multiplicity of $\lambda$. Associated to $\lambda$ we define the gap $\gamma_{\lambda}$ according to:
\begin{equation}
 \gamma_{\lambda } :=   \frac{1}{2}\min \{ |\lambda - \lambda_i  |, | \lambda - \lambda_{i+k+1}|  \} .
 \label{def:gap}
\end{equation}

Now, for every $N\in \N$, we can pick $\veps,$ $\theta $ and $\widetilde  \delta $ to be small enough so that for every $l=1, \dots, N$ we have
\[ C\left( \veps \sqrt{\lambda_l} + \frac{\widetilde \delta}{\veps}  + \veps + \theta  \right) \lambda_l \leq  \gamma_{\lambda_l}.\] 
In particular, for such choice of parameters and according to Theorem \ref{thm:evaluesveps}, with probability greater than $1- Cn \exp(-n \theta^2 \widetilde {\delta}^m)$ we have
\begin{equation}
 |\lambda_l^\veps -  \sigma_\eta \lambda_l| \leq \gamma_{\lambda_l}, \quad \forall l=1, \dots, N.  
 \label{eqn:gapineq}
 \end{equation}

Let $\lambda$ be one of $\lambda_1, \dots, \lambda_N$. By making $N$ slightly larger if necessary, we can assume that $i+k  \leq N$ (where $i$ and $k$ are as defined earlier for $\lambda$).  Let $S$ be a subspace of $L^2(\mu_n)$ spanned by eigenvectors of $\L^\veps$ with corresponding eigenvalues $\lambda_{i+1}^\veps,\dots, \lambda_{i+k}^\veps$ and let us denote by $P_{S}$ the orthogonal projection onto $S$, and by $P_S^{\perp}$ the orthogonal projection onto the orthogonal complement of $S$. We know that an arbitrary unit norm eigenfunction $f$ of $\Delta_{\rho}$ with eigenvalue $\lambda$ is at least $C^3(\M)$, and since $\Delta_{\rho} u = \lambda u$, we obviously have that $\Delta_{\rho} u$ is also $C^3(\M)$. Restricting $f$ and $\Delta_{\rho} f$ to the point cloud $X$, we can view both of these functions as elements in $L^2(\mu_n)$, and we can also see that
\[ P_S^{\perp} \Delta_{\rho} f  = \lambda P_{S}^\perp f = \lambda \sum_{j \not = i+1, \dots, i+k}  \langle  f ,  \psi_{j}^\veps \rangle_{L^2(\mu_n)} \psi_{j}^\veps \]
where $\psi_1^\veps , \dots, \psi_n^\veps$ is an orthonormal basis of eigenvectors of $\L^\veps$ with corresponding eigenvalues $\lambda_1^\veps , \dots, \lambda_n^\veps$. Likewise, we can compute $\L^\veps f$ (again thinking of $f$ as its restriction to $X$) and see that
\[ P_{S}^\perp \L^\veps f = \sum_{j \not = i+1, \dots, i+k} \lambda_{j}^\veps \langle  f ,  \psi_{j}^\veps \rangle_{L^2(\mu_n)}  \psi_j^\veps.  \]
 Subtracting these two expressions and using the orthogonality of the $\psi_j^\veps$ we obtain
 \[  \min \{  | \lambda_{i}^\veps - \lambda | , |\lambda_{i+k+1}^n - \lambda| \}  \lVert  P_{S}^\perp f \rVert_{L^2(\mu_n)} \leq  \lVert P_S^\perp (\L^\veps f -\Delta_{\rho} f)  \rVert_{L^2(\mu_n)} \leq \lVert \L^\veps f - \Delta_\rho f \rVert_{L^2(\mu_n)}. \]
 However, from \eqref{eqn:gapineq} we have 
 \[  \gamma_{\lambda }    \leq  \min \{  | \lambda_{i}^n - \lambda | , |\lambda_{i+k+1}^n - \lambda| \},\]
and so
\[  \lVert  P_{S}^\perp f \rVert_{L^2(\mu_n)} \leq \frac{1}{\gamma_{\lambda}} \lVert  \L^\veps f - \Delta_{\rho} f \rVert_{L^2(\mu_n)}.\]
Naturally, since $P_S^\perp f= f - P_S f $ we obtain
\begin{equation}
 \lVert  f-  P_{S} f \rVert_{L^2(\mu_n)} \leq \frac{1}{\gamma_{\lambda}} \lVert  \L^\veps f - \Delta_{\rho} f \rVert_{L^2(\mu_n)}.  
\label{eqn:projectionerror}
\end{equation}
On the other hand, from the pointwise consistency results from Theorem \ref{thm:con1} it follows that if $f_1, \dots, f_k$ is an orthonormal basis for the eigenspace of eigenfunctions of $\Delta_{\rho}$ with eigenvalue $\lambda$, then with probability greater than $1 - 2kn\exp\left(-c n\eps^{m+4}\right) $ we have
\[  \lVert  \L^\veps f_j - \Delta_{\rho} f_j \rVert_{L^2(\mu_n)} \leq C \veps , \quad \forall j=1,\dots,k.  \]
Combining the above with \eqref{eqn:projectionerror} we conclude that with probability greater than $1- Cn \exp(-n \theta^2 \widetilde {\delta}^m) - 2kn\exp\left(-c n\eps^{m+4}\right) $ we can find an orthonormal set $v_1, \dots, v_k$ spanning $S$ such that 
\[  \lVert  f_j - v_j   \rVert_{L^2(\mu_n)} \leq C \veps, \quad \forall j=1, \dots, k. \]
In turn this implies that if $u_1, \dots, u_k$ is a family of orthonormal eigenfunctions of $\L^\veps$ with corresponding eigenvalues $\lambda_{i+1}^\veps, \dots, \lambda_{i+k}^\veps$, then there exists an orthonormal set $\tilde f_1, \dots, \tilde f_k$ of eigenfunctions of $\Delta_{\rho}$ with eigenvalue $\lambda$ such that
\[  \lVert u_i - \tilde f_i \rVert_{L^2(\mu_n)}\leq C \veps .\]
This implies the desired result.  
\end{proof}

\begin{proof}[Proof of 1) in Theorem \ref{thm:evectorsveps}]

	In general it will not be meaningful to use the pointwise consistency results from  Theorem \ref{thm:con1} because $\veps$ here can be smaller than what the probabilistic estimates in Theorem \ref{thm:con1} allow it to be. Thus, we use an energy estimate based on Propositions \ref{prop:localnonlocal} and \ref{prop:localnonlocal} just as in \cite{trillos2018spectral}. Since this argument has been shown in detail in \cite{trillos2018spectral}, here we only present the proof for the first non-trivial eigenvectors.

	The first non-trivial eigenvalue $\lambda$ of $\Delta_{\rho}$ is equal to $\lambda_{2}, \dots, \lambda_{k+1}$ where $k$ is its multiplicity. Let $f$ be an eigenfunction of $\Delta_{\rho}$ with eigenvalue $\lambda$ and let $u$ be equal to $u= \widetilde P f$ where $\widetilde P$ is as in \eqref{def:P1}. Consider now the span of a set of orthonormal eigenvectors of $\L^\veps$ with eigenvalues $\lambda_2^n, \dots, \lambda_{k+1}^n$, and denote this linear subspace by $S$. 
	
	We see from Proposition \ref{prop:LocalNonLocal} that
	\begin{align}
	\begin{split}
  \left(1+  C\left(\frac{\widetilde \delta}{\veps} + \veps  + \theta \right) \right) \sigma_\eta \lambda_2 & = \left(1+  C\left(\frac{\widetilde \delta}{\veps} + \veps  + \theta \right) \right) \sigma_\eta D_2(f) \geq b_\veps(u) = \langle \L^\veps u , u \rangle_{L^2(\mu_n)}  
	\\ &\geq \lambda_2^\veps \lVert P_Su \rVert^2_{L^2(\mu_n)}  + \lambda_{k+2}^{\veps} \lVert u - P_S u \rVert_{L^2(\mu_n)}^2.
	\\&=  \lambda_2^\veps ( \lVert u \rVert^2_{L^2(\mu_n)} - \lVert u- P_Su \rVert^2_{L^2(\mu_n)})  + \lambda_{k+2}^{\veps} \lVert u - P_S u \rVert_{L^2(\mu_n)}^2.
	\end{split}
	\label{ineq:eigenfuncAux}
	\end{align}
	Using the estimates from the first part of Theorem \ref{thm:evaluesveps} and Proposition \ref{prop:almostisometries} we have that as long as $\widetilde \delta$ and $\theta$ are small enough, then with probability greater than $1-Cn \exp \left(-cn \theta^2 \widetilde \delta ^m \right)$, 
	\[   |\sigma_\eta \lambda_2 -\lambda_2^n| \leq C\left (\frac{\widetilde \delta }{\veps} + \veps + \theta \right)  \leq \frac{\gamma_\lambda}{2}\]
	\[   |\sigma_\eta \lambda_{k+2} -\lambda_{k+2}^n| \leq C\left (\frac{\widetilde \delta }{\veps} + \veps + \theta \right) \leq \frac{\gamma_{\lambda}}{2}  \]
	and
	\[  \lvert 1- \lVert u \rVert^2_{L^2(\mu_n)} \rvert\leq C(\theta + \widetilde \delta),\]
	where $C$ is some constant that may depend on $\lambda$ and where $\gamma_\lambda$ is the spectral gap defined in \eqref{def:gap}. Combining the above inequalities with \eqref{ineq:eigenfuncAux} we conclude that
	\[   \lVert u - P_S u \rVert_{L^2(\mu_n)}\leq \left(\frac{C}{\gamma_\lambda} \left(  \frac{\widetilde \delta}{\veps} + \veps + \theta \right) \right)^{1/2}. \]
	Now, by definition 
	\[u(x_i) - f(x_i) = \widetilde P f(x_i) - f(x_i)= n \int_{\widetilde U_i} (f(x)- f(x_i))\widetilde \rho_n(x) d\vol_\M(x).\] 
	We notice that the last term is in absolute value less than
	\[ \lVert \nabla f \rVert_{L^\infty(\mu)} \widetilde \delta, \]
	and that $ \lVert \nabla f \rVert_{L^\infty(\mu)}$ is finite because $f$ is actually $C^3(\M)$. In particular, 
	\[   \lVert  f- u  \rVert_{L^2(\mu_n)} \leq C_\M \widetilde \delta,  \]
	Therefore,
	\[ \lVert f - P_S \widetilde P f  \rVert_{L^2(\mu_n)} \leq    \left(\frac{C}{\gamma_\lambda} \left(  \frac{\widetilde \delta}{\veps} + \veps + \theta \right) \right)^{1/2} + C_{\M, \lambda} \widetilde \delta . \]
	From this it is straightforward to see that if $f_1, \dots, f_k$ form an orthonormal basis for the eigenspace of eigenfunctions of $\Delta_{\rho}$ with eigenvalue $\lambda$, then we can find an orthonormal set $v_1, \dots,  v_k$ spanning $S$ such that 
	\[  \lVert  f_j - v_j   \rVert_{L^2(\mu_n)} \leq   \left(\frac{C}{\gamma_\lambda} \left(  \frac{\widetilde \delta}{\veps} + \veps + \theta \right) \right)^{1/2} + C_{\M, \lambda} \widetilde \delta , \quad \forall j=1, \dots, k. \]
	In turn this implies that if ${u}_1, \dots, {u}_k$ is a family of orthonormal eigenfunctions of $\L^\veps$ with corresponding eigenvalues $\lambda_{2}^\veps, \dots, \lambda_{k+1}^\veps$, then there exists an orthonormal set $\widetilde f_1, \dots, \widetilde f_k$ of eigenfunctions of $\Delta_{\rho}$ with eigenvalue $\lambda$ such that
	\[  \lVert {u}_j - \widetilde f_j \rVert_{L^2(\mu_n)}\leq  \left(\frac{C}{\gamma_\lambda} \left(  \frac{\widetilde \delta}{\veps} + \veps + \theta \right) \right)^{1/2} + C_{\M, \lambda} \widetilde \delta , \quad \forall j=1, \dots, k .\]
	This implies the desired result.  
	\nc

\end{proof}

We are ready to prove 2) in Theorem \ref{thm:evaluesveps}.

\begin{proof}[Proof of 2) in Theorem \ref{thm:evaluesveps}]

%
%
%
%
%

First notice that if $\lambda>0$ is an eigenvalue of $\Delta_\rho$ and  $\lambda^\veps>0$ is an eigenvalue of the graph Laplacian $\L^\eps$, then, 
\begin{equation}\label{eq:errorest}
|\sigma_\eta \lambda - \lambda^\eps| \leq\inf_{\substack{f\in S(\Delta_\rho,\lambda)\\ u\in S(\L^\eps,\lambda^\eps)}} \frac{\lVert \L^\veps f - \sigma_\eta \Delta_\rho f \rVert_{L^2(\mu_n)}}{|\langle u,f\rangle_{L^2(\mu_n)}|}
\end{equation}
where in the above $ S(\Delta_\rho, \lambda)$ is the set of unit norm eigenfunctions of $\Delta_\rho$ with eigenvalue $\lambda$, and  $S(\L^\eps,\lambda^\eps)$ is the set of eigenvectors of $\L^\veps$ with eigenvalue $\lambda^\veps$. To see this, let $f\in S(\Delta_\rho,\lambda)$ and $u \in S(\L^\eps,\lambda^\eps)$. Then, we can restrict $f$ and $\Delta_\rho f$ to the point cloud and get
\begin{align*}
\lambda^\eps\langle u,f\rangle_{L^2(\mu_n)}&=\langle \L^\eps u,f\rangle_{L^2(\mu_n)}\\
&=\langle u,\L^\eps f\rangle_{L^2(\mu_n)}\\
&=\langle u,\L^\eps f - \sigma_\eta \lambda f + \sigma_\eta\lambda f\rangle_{L^2(\mu_n)}\\
&=\sigma_\eta \lambda \langle u,f\rangle_{L^2(\mu_n)} + \langle u,\L^\eps f - \sigma_\eta\Delta_\rho f \rangle_{L^2(\mu_n)}.
\end{align*}
Rearranging the terms we obtain
\begin{equation}\label{eq:keyforlower}
\lambda^\veps - \sigma_\eta \lambda= \frac{ \langle u,\L^\eps f - \sigma_\eta\Delta_\rho f \rangle_{L^2(\mu_n)} }{ \langle  u, f \rangle_{L^2(\mu_n)} },  
\end{equation}
provided $\langle  u, f \rangle_{L^2(\mu_n)} \neq 0$. Applying the Cauchy-Schwarz inequality we obtain
\[ |\lambda^\veps - \sigma_\eta \lambda| \leq \frac{ \lVert \L^\veps f - \sigma_\eta \Delta_\rho f \rVert_{L^2(\mu_n)} }{| \langle  u, f \rangle_{L^2(\mu_n)} |}   \]

Take $\lambda=\lambda_l$ and $\lambda^\veps =\lambda_{l}^\veps$ in the above formula. From theorem \ref{thm:evectorsveps} we know that with probability at least $1- n \exp(-n \theta^2 \widetilde {\delta}^m) - Cn\exp\left(-c n\eps^{m+4}\right) $ there are unit norm eigenvectors $u$ and $f$ of $\L^\veps$ and $\Delta_{\rho}$ with corresponding eigenvalues $\lambda_l^\veps$ and $\lambda_l$, such that
\[ \lVert u_l - f_l  \rVert_{L^2(\mu_n)} \leq C\veps, \]
so that in particular
\[  |\langle u, f \rangle_{L^2(\mu_n)}|  \geq 1- C\veps \geq \frac{1}{2},  \]
when $\veps$ is small enough. Also, from Theorem \ref{eq:con1} we know that with probability at least  $1-Cn\exp\left(-c n\eps^{m+4}\right)$ we have
\[  \lVert \L^\veps f - \sigma_\eta\Delta_\rho f \rVert_{L^2(\mu_n)} \leq C \veps.\]
The result now follows.

\end{proof}

\subsection{Undirected \texorpdfstring{$k$}{k}-NN graph.}
\label{sec:proofsUnk}

After defining appropriate interpolation and discretization maps, and establishing their relevant properties, we will be able to follow the exact same proof as the one we presented in the $\veps$-graph case. For this reason we focus on establishing Proposition \ref{prop:localnonlocalkNN} and \ref{prop:almostisometrieskNN} below and skip the rest of the details.

\nc

We will restrict our analysis to the kernel $\eta$ given by
\[  \eta(t) = \begin{cases} 1 &  t <1 \\ 0 & t>1. \end{cases} \]
Let $\psi$ be defined as \eqref{eqn:psi}. Notice that in this case we have
\begin{equation}
\psi(t) \leq \frac{1}{\sigma_\eta}\eta(t) , \quad \forall t >0.
\label{ineq:psieta}
\end{equation}
For a function $r: \M \rightarrow (0,\infty)$, and for a function $f \in L^2(\M)$ we define
\[   \Lambda_r f(x) := \frac{1}{\tau(x)} \int_{\M} K_{r(y)}(x,y) f(y) d\mu(y),  \]
as in Section \ref{sec:proofsveps} where the only difference is that now the bandwidth $r$ is allowed to change in space. The normalization constant continues to be of the form
\[  \tau(x) := \int_{\M}  K_{r(y)}(x,y) d\mu(y).\]
It is important to notice that in the definition of $\Lambda_r f$, the function $r$ is evaluated at $y$ (and not at $x$) so as to have an expression for the gradient of $\Lambda_r f$ which does not depend on derivatives of $r$ (which we may even take it to be discontinuous).

\begin{remark}
	\label{rem:emax}
In what follows we set $\veps$ to be the function defined as in \eqref{eq:epsx}. Notice that from the fact that $0 < \rho_{min} \leq\rho \leq \rho_{max}$ we see that $\veps(x)$ and $  \left( \frac{k}{n} \right)^{1/m}$ are of the same order for every $x \in \M$. Also, since $k$ has been assumed to be much smaller than $n$, then $\veps(x)^2$, which is of the same order as $(k/n)^{2/m}$, is much smaller than $\veps(x)$. We will use $\veps_{max}$ to denote the maximum of $\veps$. In particular, 
\[ c \left( \frac{k}{n}\right)^{1/m}  \leq \veps_{min} \leq \veps_{max} \leq C \left( \frac{k}{n}\right)^{1/m} .\]
\end{remark}

%
%
%
%
%

\begin{proposition}[Inequality for Dirichlet energies]
	\label{prop:localnonlocalkNN}
	Let $k$, $\widetilde{\delta}$, and $\theta$ be fixed numbers satisfying Assumptions \ref{assumptions}. Then, with probability greater than $1-  Cn\exp( -Cn \theta^2 \widetilde{\delta}^m )$ we have: \nc
	\begin{enumerate}
		\item For any $f \in L^2(\mu)$,
		\[  b_k(\widetilde P f) \leq  ( 1+ C\left(\frac{k}{n} \right)^{1/m} + C \widetilde \delta \left(\frac{n}{k} \right)^{1/m} + C\theta )\sigma_\eta D_{1-2/m}(f)  \]
		\item For any $u \in L^2(\mu_n)$,
		\[  \sigma_\eta D_{1-2/m}(\widetilde {\mathcal{I}} u) \leq ( 1+ C\left(\frac{k}{n} \right)^{1/m} + C \widetilde \delta \left(\frac{n}{k} \right)^{1/m} + C\theta ) b_k(u)  .\]
	\end{enumerate}
\end{proposition}

\begin{proposition}[Discretization and interpolation maps are almost isometries]
	Suppose $\theta, \widetilde{\delta}$ and $k$ satisfy Assumptions \ref{assumptions}. Then, with probability at least $1- C n \exp (-Cn \theta^2 \widetilde{\delta}^m )$ we have:
	\begin{enumerate}
		\item For every $f \in L^2(\mu)$,
		\[  \left  \lvert   \lVert f  \rVert_{L^2(\mu)} ^2  -  \lVert   \widetilde{P} f  \rVert_{L^2(\mu_n)}^2  \right  \rvert \leq   C \widetilde \delta \lVert f \rVert_{L^2(\mu)} \sqrt{ D(f)}  + C(\theta + \widetilde \delta) \lVert f \rVert^2_{L^2(\mu)}.  \]
		\item For every $u \in L^2(\mu_n)$,
		\[  \left  \lvert  \lVert  u  \rVert_{L^2(\mu_n)}^2   -  \lVert   \widetilde{\mathcal {I}} u  \rVert_{L^2(\mu)}^2   \right \rvert \leq  C \left( \frac{k}{n}\right)^{1/m} \lVert u  \rVert_{L^2(\mu)} \sqrt{b(u)} + C(\theta + \widetilde \delta) \lVert u \rVert^2_{L^2(\mu_n)}.  \]	
	\end{enumerate}
	\label{prop:almostisometrieskNN}
\end{proposition}
\nc

In order to show the above propositions we will first need to introduce some non-local energies $E_r$ and $\widetilde E _r$ (we will use the same notation for simplicity as in the $\veps$-graph case) defined in terms of the spatially varying length-scale $r: \M \rightarrow (0,\infty)$ by
\[E_r (f) := \int_\M \int_\M    \eta\left( \frac{d_\M(x,y) }{r(y)}\right) \frac{(f(x) - f(y) )^2}{ r(y)^{m+2} \rho(y)^{2/m-1}} d\vol_\M(x) d\vol_\M(y) , \quad f \in L^2(\mu) \]
\[\widetilde E_r (f) := \int_\M \int_\M     \eta\left( \frac{d_\M(x,y) }{r(y)}\right)\frac{(f(x) - f(y) )^2}{ r(y)^{m+2} \widetilde \rho_n(y)^{2/m-1}}  d\vol_\M(x) d\vol_\M(y) , \quad f \in L^2(\mu), \]
where $\widetilde \rho_n$ is the density from Proposition \ref{prop:AuxiliaryDensity}. Notice that the $L^\infty$ bound for the difference between $\rho$ and $\widetilde \rho_n$ implies
\begin{equation}
(1- C(\theta + \widetilde \delta)) E_r(f)    \leq  \widetilde E_r(f) \leq (1+  C(\theta + \widetilde {\delta}) ) E_r (f), \quad \forall f \in L^2(\mu), 
\label{ineq:tildenotildekNN}
\end{equation}
These inequalities are analogous to the ones in \eqref{lemma:NonlocalIneqs}.

The following three lemmas will be the main tools for proving Propositions \ref{prop:localnonlocalkNN} and \ref{prop:almostisometrieskNN}.

\begin{lemma}
	Suppose that $k$, $\widetilde \delta$ and $\theta$ satisfy Assumptions \ref{assumptions}, and let $\veps: \M \rightarrow (0,\infty)$ be as in \eqref{eq:eps}.
	Then, with probability at least $1- Cn \exp(-Cn \theta^2 \widetilde \delta^m)$ we have for all $i=1, \dots, n$
	\begin{equation}
	\veps(x_i)-  \widetilde C (\veps_{max}^2 +  \widetilde\delta \veps_{max}+ \theta \veps_{max}) \leq  \veps_k(x_i) \leq \veps(x_i)+  \widetilde C (\veps_{max}^2 +  \widetilde\delta \veps_{max}+ \theta \veps_{max}).   
	\label{ineq:vepsk}
	\end{equation}
	
	\label{lemma:evpskeps}
\end{lemma}
\begin{proof}
	We prove the upper bound. The lower bound is proved similarly.
	
	Let $\widetilde \rho_n$ be the density from Proposition \ref{prop:AuxiliaryDensity}, which exists with probability at least $1- Cn \exp(-n \theta^2 \widetilde \delta^m)$. First of all for $x \in \M$ and $0< \gamma_1 < \gamma_2$ small enough define
	\[ A(x, \gamma_1, \gamma_2):= B_\M(x, \gamma_2) \setminus B_\M(x, \gamma_1). \]
	Then, we see that
	\[  C_1 \gamma_1^{m-1}(\gamma_2 - \gamma_1 )  \leq \widetilde \mu_n( A(x,\gamma_1, \gamma_2 ) \leq C_2 \gamma_1^{m-1}(\gamma_2 - \gamma_1 )  \]
	where $C_1, C_2$ do not depend on $x$ nor $\gamma_1, \gamma_2$ as long as these numbers are small enough.
	
	Pick $\gamma_2 = \veps(x)+ \widetilde C (\veps(x)^2 + \widetilde \delta + \theta )$ for a large enough constant $\widetilde C$ that will be chosen soon. We also let $\gamma_1 = \veps(x)$. Then,  
	\begin{align*}
	\widetilde{\mu}_n \left(B_\M(x_i, \gamma_2) \right) &\geq  \int_{B_\M(x_i, \veps(x_i))}     \widetilde  \rho_n(x)  d\vol_\M(x) + \widetilde{\mu}_n (A(x, \gamma_1, \gamma_2))
	\\&  \geq (1- C\veps(x_i) )(1- C(\theta + \delta ))\rho(x_i) \vol_\M(B_\M(x_i, \veps(x_i))) + \widetilde{\mu}_n (A(x, \gamma_1, \gamma_2))
	\\ &\geq   (1- C \veps(x_i))(1- C(\theta + \delta ))(1- CK\veps(x_i)^{m+2}) \rho(x_i) \alpha_m (\veps(x_i))^m + \widetilde{\mu}_n (A(x, \gamma_1, \gamma_2))
	\\& \geq   (1- C \veps(x_i))(1- C(\theta + \delta ))(1- CK\veps(x_i)^{m+2}) \frac{k}{n} + \widetilde{\mu}_n (A(x, \gamma_1, \gamma_2))
	\\ & \geq   \frac{k}{n} -  C (\veps(x_i) + \widetilde \delta + \theta)\veps(x_i)^{m}  + \widetilde{\mu}_n (A(x, \gamma_1, \gamma_2))
	\\&  \geq   \frac{k}{n} -  C (\veps(x_i) + \widetilde \delta + \theta)\veps(x_i)^{m}     + C_1 \veps(x_i)^{m-1}(\widetilde C (  \veps(x_i)^2 + \widetilde \delta \veps(x_i) + \theta \veps(x_i) ) ).
	\end{align*}
	We pick $\widetilde C$ precisely so that the above is greater than $k/n$. Finally, we see that
	\[ \mu_n(  B_\M(x_i, \gamma_2 + \widetilde \delta ) ) = \widetilde \mu_n \left( \widetilde{T}^{-1}(B_\M(x_i, \gamma_2 + \widetilde \delta)) \right) \geq \widetilde \mu_n(B_\M(x_i, \gamma_2 )) \geq k/n,   \]
	from where it now follows that $\veps_k(x_i)\leq \gamma_2 + \widetilde \delta$.

\end{proof}

\begin{lemma}
Let $r: \M \rightarrow (0,\infty)$ be an arbitrary function which is  bounded away from zero and for which $r_{max}$ is sufficiently small. Then, for any $f \in L^2(\mu)$ we have
	\[  E_r(f) \leq (1+ Cr_{max})(1+\beta_r)^{2m+4} \sigma_\eta   D_{1-2/m}(f).\] 
	Where in the above $\beta_r$ is defined as
	\begin{equation}
	 \beta_r := \sup_{x,y \in \M, d_\M(x,y) \leq r_{max}}\frac{r(x)-r(y)}{r_{min}}  
	 \label{eqn:betar}
	\end{equation}
	
	\label{nonlocalLessLocal}
\end{lemma}

As we will later see, in our setting we will work with $r$ for which $r_{max}, \beta_r \ll 1,$  so that Lemma \ref{nonlocalLessLocal} implies that $E_r(f)$ is smaller than a quantity that up to leading order is equal to $\sigma_\eta D_{1-2/m}(f)$.

\begin{proof}
	
By density of smooth functions in $H^1(\mu)$ it is enough to prove the result for $f$ smooth.

In what follows we will use the geodesic flow $\Phi_t: T \M \rightarrow T \M$, which maps a point $(x,v) \in \M \times T_x \M$ in $\M$'s tangent bundle into the point $\Phi_t(x,v) =(\exp_x(tv), d\exp_x(v))$ (i.e. flow $x$ for $t$ seconds along the geodesic emanating from it with initial velocity $v$). We will use $\xi=(x,v)$ to represent a generic point in the tangent bundle and abuse notation slightly to write things like $ g(\xi)=g(x)$, when $g$ is a real valued function on $\M$. We will also use $df$ (i.e. the differential of $f$) which is the 1-form that when acting on a tangent vector $v$ returns the directional derivative of $f$ in the direction $v$, and will write things like $df(\xi)$ to denote the directional derivative of $f$ at the point $x$ in the direction $v$.

With this notation in hand, and following the proof of Lemma 3.3 in \cite{BIK}, we obtain for every $x \in \M$ 
\[  \int_{B_\M(x, r(x))} |f(y)- f(x)|^2 d\vol_\M(y) \leq \int_{B(0, r(x)) \subseteq T_x \M}\int_{0}^1 |df(\Phi_t(x,v))|^2 dt dv.     \]
Dividing by $r(x)^{m+2}  \rho(x)^{2/m-1}$ and integrating over $x$ we get:
\begin{align}
\begin{split}
 &\int_\M \int_{B_\M(y, r(y))} \frac{|f(y)- f(x)|^2}{r(x)^{m+2}  \rho(x)^{2/m-1}} d\vol_\M(y) d \vol_\M(x)
 \\ & \leq (1+\beta_r)^{m+2} (1+ C  r_{max}) \int_{\M}\int_{B_m(0, r(x)) \subseteq T_x \M}\int_{0}^1 \frac{|df(\Phi_t(y,v))|^2}{r(\Phi_t(x,v))^{m+2}  \rho(\Phi_t(x,v))^{2/m-1}} dt dv d\vol_\M(x)
  \\& = (1+\beta_r)^{m+2} (1+ C r_{max}) \int_{0}^1 \int_{\mathcal{B}_r} \frac{|df(\Phi_t(\xi  ))  |^2}{r(\Phi_t(\xi))^{m+2}  \rho(\Phi_t(\xi))^{2/m-1}}  d\vol_{T\M}(\xi) dt, 
  \end{split}
  \label{eqn:auxNonlocalknn}
\end{align}
where for the first inequality we have used the definition of $\beta_r$ and the fact that for every $y$ with $d_\M(x,y) \leq r_{max}$ we have 
\[\left(  \frac{\rho(y)}{\rho(x)}\right)^{2/m-1} =  \left(  1+ \frac{\rho(y)-\rho(x)}{\rho(x)}\right)^{2/m-1}\leq 1+ Cr_{max}. \]
Finally, in the last line we use $\vol_{T\M}$ to denote the volume form on $T\M$, and
\[  \mathcal{B}_r := \{  (x,v) \in T\M \: : \:  x \in \M , \quad v\in B_m(0,r(x)) \subseteq T_x\M  \}.    \]

Now let $t \in (0,1)$ and let $(x,v)  \in \mathcal{B}_r$. Define $\tilde{x}:= \exp_{x}(-tv)$ and $\tilde v := d\exp_x(-v)$. It is straightforward to see that 
\[  \Phi_t(\tilde x, -\tilde v) = (x,v).\]  
Moreover, 
\[ \lVert -\tilde v  \rVert_{\tilde x} =  \lVert d \exp_{x}(-v)  \rVert_{\tilde x} = \lVert  -v   \rVert_{x}\leq  r(x) = r( \tilde x) + (r(x) - \tilde r (x)) \leq \hat{r}(\tilde x), \]
where 
\[\hat{r}(\tilde x):= (1+  \alpha_r) r( \tilde x).\]
We have shown that if $(x,v) \in \mathcal{B}_r$, then for every $t \in (0,1)$ $(x,v)\in \Phi_t\left(\mathcal{B}_{\hat{r}}\right)$. That is,
\[ \mathcal{B}_r \subseteq \Phi_t\left(\mathcal{B}_{\hat{r}}\right).  \]
From this we deduce that for all $t \in (0,1)$ we have
 \[  \int_{\mathcal{B}_r} \frac{|df(\Phi_t(\xi  ))  |^2}{r(\Phi_t(\xi))^{m+2}  \rho(\Phi_t(\xi))^{2/m-1}}  d\vol_{T\M}(\xi)   \leq  \int_{\Phi_t(\mathcal{B}_{\hat{r}})} \frac{|df(\Phi_t(\xi  ))  |^2}{r(\Phi_t(\xi))^{m+2}  \rho(\Phi_t(\xi))^{2/m-1}} d\vol_{T\M}(\xi)   \]
In turn, the right hand side is equal to  
\[ \int_{\mathcal{B}_{\hat{r}}} \frac{|df(\Phi_t(\xi  ))  |^2}{r(\xi)^{m+2}  \rho(\xi)^{2/m-1}} d\vol_{T\M}(\xi)  \]
which follows from the well known fact that $\Phi_t$ preserves $\vol_{T\M}$ (i.e. it pushes forward $\vol_{T\M}$ into itself ). Integrating over $t$ and using \eqref{eqn:auxNonlocalknn} we deduce that
\[    E_r(f) \leq  (1+\beta_r)^{m+2} (1+ Cr_{max})  \int_{\M} \frac{1}{r(x)^{m+2}\rho(x)^{2/m -1} }\left( \int_{B_m(0, \hat{r}(x))}|df(x,v)  |^2 dv \right) d\vol_{\M}(x). \]    
Finally, a simple computation shows that
\[ \int_{B_m(0, \hat{r}(y))}|df(x,v)  |^2 dv=  \int_{B_m(0, \hat{r}(y))}|\langle \nabla f(x) , v \rangle|^2 dv = \sigma_\eta |\nabla f(x)|^2 \hat{r}(x)^{m+2}, \]
and so
\begin{align*}
E_r(f) & \leq  (1+\beta_r)^{m+2}  (1+ Cr_{max})  \sigma_\eta \int_{\M} \left( \frac{\hat{r}(x)}{r(x)} \right)^{m+2} |\nabla f(x)|^2\rho(x)^{1-2/m} d\vol_\M(x)
\\& \leq (1+\beta_r)^{2m+4}  (1+ Cr_{max}) \sigma_\eta D_{1-2/m}(f).
\end{align*}
This concludes the proof.

\end{proof}

\begin{lemma}
	Let $r: \M \rightarrow (0,\infty)$ be an arbitrary function which is  bounded away from zero and for which $r_{max}$ and $\beta_r$ (as defined in \eqref{eqn:betar}) are sufficiently small. Then, for any $f \in L^2(\mu)$ we have
	\[  \sigma_\eta D_{1-2/m}(\Lambda_r f) \leq (1+ C\beta_r^{2m+4} + Cr_max^2  )  E_{\hat{r}}(f),\]
	where the length-scale $\hat{r}$ in $E_{\hat{r}}(f)$ is of the form
	\[  \hat{r}(x)= (1+\beta_r)r(x).   \]
\label{lemma:localLessnonlocal}
\end{lemma}

\begin{proof}

We first notice that for a fixed function $f \in L^2(\mu)$, the gradient of the function $\Lambda_r f$ can be written as
\[\nabla \Lambda_r f (x) = \tau^{-1}(x) A_1(x) +A_2(x), \] 
where 
\[ A_1(x) \coloneqq \int_{\mathcal{A}(x)} \nabla K_r(\cdot,y)(x) (f(y)-f(x)) d\vol_\M(y),\]
\[ A_2(x) \coloneqq \nabla (\tau^{-1})(x)\int_{\mathcal{A}(x)} K_{r(y)}(x,y) (f(y)-f(x)) d\vol_\M(y),\]
and the region $\mathcal{A}(x)$ is defined as
\[ \mathcal{A}(x):= \{y \in \M \: : \: \lvert x-y \rvert \leq r(y). \}  \]
\nc

First $\abs{A_1(x)}= \applied{A_1(x)}{w}$ for some unit vector $w\in T_x\M$. 
Therefore, 
\begin{align*}
\abs{A_1(x)}&= \lvert \applied{ A_1(x)}{ w} \rvert \\
&= \left \lvert  \int_{\mathcal{A}(x)} \frac{1}{\sigma_\eta r^{m+2}(y) } \eta\left( \frac{d_\M(x,y)}{r(y)} \right) (f(y) - f(x)) \applied{\exp_x^{-1}(y)}{w} d\vol_\M(y) \right \rvert \\ 
&=  \left \vert \int_{\mathcal{A}}\frac{1}{\sigma_\eta r^{m+2}(v)} \eta\left(\frac{\abs{v}}{r(v)}\right) \varphi(v) \applied{v}{w} J_x(v) dv \right \rvert \\
&\leq   \int_{\mathcal{A}}\frac{1}{\sigma_\eta r^{m+2}(v)} \eta\left(\frac{\abs{v}}{r(v)}\right) \lvert \varphi(v) \rvert  \lvert \applied{v}{w} \rvert J_x(v) dv 
\end{align*}
where $\varphi(v) \coloneqq f(\exp_x(v))-f(x)$, $\mathcal{A}:= \{ v \in T_x \M \: : \: |x- \exp_x(v) | \leq r(v) \}$ and with a slight abuse of notation we write $r(v)=r(\exp_x(v))$ and also $\hat r(v)= \hat r(\exp_x(v))$. Notice that by definition of $\hat{r}$ we have
\[ \hat r_0 := \hat{r}(0) \geq r(v) \]
and so
\[ \abs{A_1(x)}  \leq \left( 1+ \beta_r \right)^{m+2}  \int_{B(0,\hat{r}_0)}\frac{1}{\sigma_\eta \hat{r}_0^{m+2}} \eta\left(\frac{\abs{v}}{\hat{r}_0}\right) \lvert \varphi(v) \rvert  \lvert \applied{v}{w} \rvert J_x(v) dv.  \]
By the Cauchy-Schwartz inequality,
\begin{align*}
\frac{\abs{A_1(x)}^2}{\left( 1+ \beta_r \right)^{2m+4}} &\leq \left( \frac{1}{\sigma_\eta \hat{r}_0^{m+2}}\right)^2  \left( \int_{B(0,\hat r_0)} \abs{\varphi(v)}^2 J_x(v)^2 \eta \left(\frac{\abs{v}}{\hat r_0}\right) dv \right) \left( \int_{B(0,\hat r_0)} \langle v,w \rangle ^2 \eta \left(\frac{\abs{v}}{ \hat r_0}\right) dv \right) \\
& = \frac{1}{\sigma_\eta {\hat{r}_0}^{m+2}}\left( \int_{B(0,\hat{r}_0)} \abs{\varphi(v)}^2 J_x(v)^2 \eta \left(\frac{\abs{v}}{\hat{r}_0}\right) dv \right) 
\end{align*}
where, in the last step, we used radial symmetry to conclude that
\[ \int_{B(0,\hat{r}_0)} \applied{ v}{w}^2 \eta \left(\frac{\abs{v}}{\hat{r}_0}\right) dv  = \hat{r}_0^{m+2} \int_{B(0,1)} v_1^2 \eta(\abs{v}) dv = \hat{r}_0^{m+2} \sigma_\eta.\]
We obtain,
\begin{align*}
\frac{\abs{A_1(x)}^2}{\left( 1+ \beta_r \right)^{2m+4}} &\leq \frac{1+ C(r_{max}(1+ \beta_r))^2}{\sigma_\eta \hat{r}_0^{m+2}} \int_{B(0,\hat r_0)} \abs{\varphi(v)}^2\eta \left(\frac{\abs{v}}{\hat r_0}\right) J_x(v)  dv \\
&= \frac{1+ C(r_{max}(1+ \beta_r))^2}{\sigma_\eta \hat{r}(x)^{m+2}} \int_\M \eta\left(\frac{d_\M(x,y)}{\hat r(x)}\right) (f(y)-f(x))^2 d\vol_\M(y)
\end{align*}
\nc
Integrating this inequality with respect to $\rho^{1-2/m}(x) d \vol_\M(x)$ and using the Lipschitz continuity of $\rho$, we obtain 
\begin{align}
\begin{split}
\norm{A_1}_{L^2(\M,\rho^{1-2/m}\vol_\M)}^2 
&\leq { \left( 1+ \beta_r \right)^{2m+4} (1+ C(r_{max}(1+ \beta_r))^2)} 
\\& \cdot \frac{1}{{\sigma_\eta }}  \int_\M \int_{\M} \eta\left(\frac{d_\M(x,y)}{\hat r(x)}\right) \frac{\abs{f(y)-f(x)}^2}{\hat{r}(x)^{m+2} \rho(x)^{2/m-1}} d \vol_\M(y)  d\vol_\M(x) \\
&=  { \left( 1+ \beta_r \right)^{2m+4} (1+ C(r_{max}(1+ \beta_r))^2)} \frac{1}{\sigma_\eta}E_{\hat{r}}(f).
\end{split}
\label{eqn:A1ineq}
\end{align} 
Regarding $A_2$, first note that by Lemma 5.1 in \cite{BIK} we have
\begin{equation}
  (1+ C r_{max}^2)^{-1} \leq  \tau(x) \leq (1+ Cr_{max}^2)
  \label{eqn:tauineqs} 
\end{equation}
and also
\begin{equation*}
\abs{\nabla (\tau^{-1})}\leq C r_{max}
\label{ineq:gradtau}
\end{equation*}
Therefore, 
\begin{align*}
\abs{A_2(x)}^2 &\leq \abs{\nabla (\tau^{-1})(x)}^2 \tau(x) \int_\M \abs{f(y)-f(x)}^2K_{r(y)}(x,y) d\vol_\M(y) \\
&\leq   r_{max}^2 \int_\M \frac{1}{r(y)^m}\eta\left(\frac{d_\M(x,y)}{r(y)}\right) \abs{f(y)-f(x)}^2 d \vol_\M(y),
\end{align*}
where the first inequality follows from Cauchy-Schwartz inequality and in the second one we used 
\[ \psi(s) \leq \frac{1}{\sigma_\eta}\eta(s),\quad \forall s>0. \]
Integrating the estimate for $\lvert A_2(x)\rvert^2$ with respect to $\rho^{1-2/m}(x) d \vol_\M(x)$ while using the Lipschitz continuity of $\rho$ we obtain
\[  \norm{A_2}_{L^2(\M,p^{1-2/m}\vol_\M)}^2 \leq  C(1+ r_{max})^2 r_{max}^4 E_r(f)  \leq  C(1+ \beta_r)^{m+2}(1+ r_{max})^2 r_{max}^4 E_{\hat r}(f). \]
By combining the above with \eqref{eqn:A1ineq} and \eqref{eqn:tauineqs} estimates and the lower bound for $\tau$ we obtain the desired inequality.
\end{proof}

We will apply Lemma \ref{lemma:localLessnonlocal} with $r$ for which $r_{max}, \beta_r \ll 1,$. In particular $E_{\hat r}(f)$ (where $\hat r$ is to the first order equal to $r$) is greater than a quantity that up to leading order is equal to $\sigma_\eta D_{1-2/m}(f)$. This is the reverse inequality to the one in Lemma \ref{nonlocalLessLocal}.

\begin{proof}[Proof of Proposition \ref{prop:localnonlocalkNN}]
	
%

To prove the first inequality, let us start by recalling that $r_k(x_i,x_j) = \max\{ \veps_k(x_i), \veps_k(x_j) \}$ where $\veps_k$ is defined as in \eqref{eq:epsk}. Recall also that for arbitrary $u \in L^2(\mu_n)$ we can write
\[b_k(u):= \frac{1}{n^2} \left( \frac{n\alpha_m}{k} \right)^{1+2/m} \sum_{i,j} \eta\left(\frac{|x_i-x_j|}{r_k(x_i, x_j)} \right) ( u(x_i) - u(x_j) )^2.\]
Using the $\infty$-OT map $\widetilde T$ between $\widetilde \mu_n$ and $\mu_n$ we can apply the change of variables formula to rewrite $b_k(u)$ as
\begin{align*}
\begin{split}
 b_k(u) & =  \left( \frac{n\alpha_m}{k} \right)^{1+2/m} \int_\M\int_{\M} \eta \left(\frac{|\widetilde T(x)- \widetilde T (y)|}{\max \{ \veps_k(\widetilde T (x)), \veps_k(\widetilde{T}(y))   \}} \right) |   \widetilde P^* u (x) -  \widetilde P^* u (y) |^2  d\widetilde \mu_n(y)d\widetilde \mu_n(x)   
 \\& = \int_\M\int_{\M}  \frac{1}{\rho(\widetilde T (y))^{1+2/m}\veps(\widetilde T (y))^{m+2}}\eta \left(\frac{|\widetilde T(x)- \widetilde T (y)|}{\max \{ \veps_k(\widetilde T (x)), \veps_k(\widetilde{T}(y))   \}} \right) |   \widetilde P^* u (x) -  \widetilde P^* u (y) |^2  d\widetilde \mu_n(y)d\widetilde \mu_n(x).
 \end{split}
\end{align*}
where for the last line we have used the definition of $\veps$ in \eqref{eq:eps}. From Lemma \ref{lemma:evpskeps} we know that 
\[   \veps_k(\widetilde T (x))  \leq  \veps(\widetilde T(x)) + C  \veps_{max}^2 + C\veps_{max}\widetilde \delta + C \veps_{max} \theta. \]
Moreover, using the smoothness of $\rho$ and the definition of $\veps$ for $x,y$ with $|x-y|\leq \veps_{max}$ we have
\[  \veps(\widetilde T x) \leq \veps( \widetilde T y)  + C(\veps_{max}^2 + \widetilde \delta \veps_{max}).   \]
Thus, if $|\widetilde T (x) - \widetilde T(y)| \leq r_k(\widetilde T(x) , \widetilde T (y)) $, we have $|\widetilde T (x) - \widetilde T(y)| \leq  \veps( \widetilde T y)  + C(\veps_{max}^2 + \widetilde \delta \veps_{max}) $ and also
\[  |x-y|\leq  d_\M(x,y) \leq \veps( \widetilde T y)  + C(\veps_{max}^2 + \widetilde \delta \veps_{max} + C \theta \veps_{max}) + 2 \widetilde \delta=: r(y).\]
It follows that 
\begin{align*}
& b_k(u)  \leq \int_\M\int_{\M}  \frac{1}{\rho(\widetilde T (y))^{1+2/m}\veps(\widetilde T (y))^{m+2}}\eta \left(\frac{|\widetilde T(x)- \widetilde T (y)|}{ r(y)  } \right) |   \widetilde P^* u (x) -  \widetilde P^* u (y) |^2  d\widetilde \mu_n(y)d\widetilde \mu_n(x) 
 \\&\leq (1+ C \widetilde \delta )( 1+ C(\veps_{max}^2 + \widetilde \delta + C \theta \veps_{max})  )
 \\& \cdot   \int_\M\int_{\M}  \frac{1}{\rho(y)^{2/m-1}r(y)^{m+2}}\eta \left(\frac{|\widetilde T(x)- \widetilde T (y)|}{ r(y)  } \right) |   \widetilde P^* u (x) -  \widetilde P^* u (y) |^2  d\vol_\M(y)d\vol_\M(x)
 \\&= (1+ C \widetilde \delta )( 1+ C(\veps_{max} + \frac{\widetilde \delta}{\veps_{min}} + \theta ))E_r(\widetilde P^* u).
 \\ 
\end{align*}
Let $f \in L^2(\mu)$ and let $u:= \widetilde P f$ in the above estimate to obtain  
\[ b_k(\widetilde P f) \leq  (1+ C \widetilde \delta )( 1+ C(\veps_{max} + \frac{\widetilde \delta}{\veps_{min}} + \theta ))E_r(\widetilde P^* \widetilde P f) \leq (1+ C \widetilde \delta )( 1+ C(\veps_{max} + \frac{\widetilde \delta}{\veps_{min}} + \theta ))E_r(f), \]
where in the last line we have used Jensen's inequality to relate $E_r(\widetilde P^* \widetilde P f)$ and $E_r(f)$. On the other hand, from Lemma \ref{nonlocalLessLocal} we have
\[ E_r(f) \leq (1+ Cr_{max})(1+\beta_r)^{2m+4} \sigma_\eta   D_{1-2/m}(f).\] 
However for the function $r$ defined above the quantity $\beta_r$ satisfies 
\[ \beta_r \leq  C \veps_{max}, \]
so
\[   b_k(\widetilde P f) \leq ( 1+ C(\veps_{max} + \frac{\widetilde \delta}{\veps_{min}} + \theta ))\sigma_\eta D_{1-2/m}(f),  \]
proving in this way the desired inequality thanks to Remark \ref{rem:emax}.

The second inequality is proved very similarly using the lower bound in Lemma \eqref{ineq:vepsk}, which allows us to show that
\begin{equation}
E_{r}(\widetilde P ^* u) \leq (1+ C(\veps_{max} + \frac{\widetilde \delta}{\veps_{min}} + \theta))b_k(u), \quad \forall u \in L^2(\mu_n), 
\label{ineq:kNNnonlocaldiscrete}
\end{equation}
for $r$ of the form $r(y) = \veps(y)- C\veps_{max}^2 - C\widetilde \delta  - C \theta \veps_{max} $. Lemma \ref{lemma:localLessnonlocal} in turn allows us to bound $\sigma_\eta D_{1-2/m}(\widetilde \I u ) $ with $E_{r}(\widetilde P ^* u)$  (up to leading order).

\nc

%

\end{proof}

\begin{proof}[Proof of Proposition \ref{prop:almostisometrieskNN}]
	The first part follows directly from 1) in Proposition \ref{prop:LocalNonLocal} since 
	\[ D_2(f) \leq  C D_{1-2/m}(f), \]
	given hat we have assumed that $\rho$ is bounded and also bounded away from zero.

	For the second part, notice that for every $f \in L^2(\mu)$ by Jensen's inequality we have
	\begin{align*}
	\begin{split}
	(\Lambda_r f(x) - f(x))^2   & =  \left(\frac{1}{\tau(x)} \int_{\M} \frac{1}{r(y)^m}\psi\left(\frac{|x-y|}{r(y)} \right)(f(y) - f(x))\rho(y) d\vol_\M(y)\right)^2
	\\& \leq \frac{1}{\tau(x)} \int_{\M} \frac{1}{r(y)^m} \psi\left(\frac{|x-y|}{r(y)} \right)(f(y) - f(x))^2\rho(y) d\vol_\M(y)
	\\& \leq C \int_{\M} \frac{1}{r(y)^m} \psi\left(\frac{|x-y|}{r(y)} \right)(f(y) - f(x))^2\rho(y) d\vol_\M(y)
	\\& \leq C \int_{\M} \frac{1}{r(y)^m} \eta\left(\frac{|x-y|}{r(y)} \right)(f(y) - f(x))^2\rho(y) d\vol_\M(y)
	\end{split}
	\end{align*}
	where the last inequality follows from \eqref{ineq:psieta}. Multiplying by $\rho(x) $ and integrating over all $x \in \M$ we get
	\begin{equation}
	\lVert \Lambda_r f - f  \rVert^2_{L^2(\mu)} \leq C  r_{max}^2 E_r(f),
	\label{eqn:LambdarkNN}
	\end{equation}
	where we have used again the fact that $\rho$ is bounded and bounded away from zero. Taking $r$ to be 
	\[ r(y):= \veps(y) - 2 \widetilde \delta   \] 
	we continue as in the second part of the proof of \eqref{prop:almostisometries} to obtain
	\begin{align*}
	\begin{split}
	\left \lvert \lVert \widetilde \I u \rVert _{L^2(\widetilde \mu _n)}   -  \lVert u \rVert _{L^2(\mu_n)}   \right \rvert & \leq   \lVert \Lambda_{\veps - 2\widetilde \delta} \widetilde P^* u  - \widetilde P^* u    \rVert _{L^2(\widetilde \mu _n)} 
	\\&  \leq \left( 1+ \frac{\lVert \rho - \widetilde \rho_n \rVert_\infty}{\rho_{min}} \right) \cdot \lVert \Lambda_{\veps - 2\widetilde \delta} \widetilde P^* u  - \widetilde P^* u    \rVert _{L^2(\mu)} 
	\\&  \leq  \left( 1+ \frac{\lVert \rho - \widetilde \rho_n \rVert_\infty}{\rho_{min}} \right) \cdot C \veps_{max} \sqrt{ E_{\veps - 2 \widetilde \delta}( \widetilde P^* u) }
	\\& \leq  C  \veps_{max} \sqrt{b_k(u)},
	\end{split}
	\end{align*}
	where for the third inequality we have used \eqref{eqn:LambdarkNN}, and for the last one we have used  \eqref{ineq:kNNnonlocaldiscrete}. Also, from \eqref{eqn:LambdarkNN} 
	\[   \lVert \widetilde I   u  \rVert_{L^2(\widetilde \mu_n)} = \lVert \Lambda_{r} \widetilde P^* u  \rVert_{L^2(\widetilde \mu_n)} \leq  C \lVert \widetilde P^* u \rVert_{L^2(\widetilde \mu_n)} = C \lVert u \rVert_{L^2(\mu_n)},   \]
 and hence
 \[  \left \lvert \lVert \widetilde \I u \rVert _{L^2(\widetilde \mu _n)}^2   -  \lVert u \rVert _{L^2(\mu_n)}^2   \right \rvert \leq C \veps_{max} \lVert u \rVert_{L^2(\mu_n)}\sqrt{b_k(u)}.   \]
 Finally, we use \eqref{ineq:intrhontilderhon} to compare $\lVert \widetilde \I u \rVert^2 _{L^2(\widetilde \mu _n)} $ and $\lVert \widetilde \I u \rVert^2 _{L^2( \mu)} $.

\end{proof}

With Proposition \ref{prop:localnonlocalkNN} and \ref{prop:almostisometrieskNN} in hand, the proofs of Theorems \ref{thm:evaluesks} and \ref{thm:evectorsks} are exactly as in the $\veps$-setting and so we skip the details.

\section{Convergence of eigenvectors in \texorpdfstring{$TL^2$}{Tl2} and convergence of graph Laplacian embeddings}	
\label{sec:TL2convergence}

To establish Theorem \ref{cor:WassersteinIneq} it will be convenient to recall the definition of the $TL^2$ space presented in \cite{trillos2018variational}. Here we consider $\R^L$-valued functions. 

We define the set
\[ TL^2(\M; \R^L):= \left\{  (\gamma, H) \: : \: \gamma \in \mathcal{P}(\M) , \quad H \in L^2(\gamma; \R^L)   \right\},  \]
and the metric
\[ \left(d_{TL^2}( (\gamma, H) , (\widetilde \gamma,  \widetilde H)  ) \right)^2 := \min _{\pi\in \Gamma(\gamma, \widetilde \gamma)} \int_{\M \times \M}d_\M(x,y)^2 d \pi(x,y)  + \int_{\R^L \times \R^L} | H(x) - \widetilde{H}(y)|^2 d \pi(x,y).   \]
In the above $\mathcal{P}(\M)$ denotes the space of Borel probability measures on $\M$, and for $\gamma \in \mathcal{P}(\M)$, $L^2(\gamma;\R^L)$ denotes the $\R^L$-valued $L^2$ functions with respect to $\gamma$. $\Gamma(\gamma, \widetilde \gamma)$ denotes the set of couplings or transport plans between $\gamma$ and $\widetilde \gamma$. Notice that if we remove the second term on the right hand side of the definition of $d_{TL^2}$ we get,
\[W_\M^2(\gamma, \widetilde \gamma) := \min _{\pi\in \Gamma(\gamma, \widetilde \gamma)} \int_{\M \times \M}d_\M(x,y)^2 d \pi(x,y) \] 
which is the square of the 2-Wasserstein distance between $\gamma$ and $\widetilde \gamma$ (using the intrinsic geometry of $\M$).

For our purposes, the key property of the $TL^2$-metric is that it gives an upper bound for the Wasserstein distance between induced embeddings. 
\begin{lemma}
	\label{WassersteinIneq}
	Let $(\gamma,H),  ( \widetilde \gamma ,\widetilde{H}) \in TL^2(\M; \R^L)$. Then for all $\pi \in \Gamma(\gamma, \widetilde \gamma)$ we have
	\[ W_2^2(H_\sharp \gamma , \widetilde H _\sharp \widetilde \gamma) \leq \int_{\R^L \times \R^L} | H(x) - \widetilde{H}(y)|^2 d \pi(x,y).  \]
	In particular,
	\[  W_2( H_{\sharp } \mu , \widetilde{H}_{\sharp} \widetilde \gamma   )  \leq  d_{TL^2}( (\gamma, H) , (\widetilde \gamma,  \widetilde H))  .  \]
\end{lemma}
\begin{proof}
To see this let $\pi \in \Gamma(\gamma, \widetilde \gamma)$ and define
\[ \hat{\pi}:= (H \times \widetilde H)_{\sharp} \pi,  \]
where $H \times \widetilde H$ is the map
\[ H \times \widetilde H : (x,y) \in \M \times \M \longmapsto (\hat{x},\hat{y})=(H(x), \widetilde H(y)) \in \R^L \times \R^L. \]
It is straightforward to see that $\hat \pi \in \Gamma( H_\sharp \gamma, \widetilde H_\sharp \widetilde \gamma )$. From the change of variables formula we see that
\[ W_2^2(  H_{\sharp } \gamma  , \widetilde{H}_{\sharp} \widetilde \gamma) \leq \int_{\R^L \times \R^L} |\hat x- \hat y|^2 d\hat{\pi}(\hat x, \hat y) = \int_{\M \times \M} |H(x)-\widetilde  H(y)|^2 d\pi(x, y),    \]
which implies the desired result.

\end{proof}

\begin{proof}[Proof of Theorem \ref{cor:WassersteinIneq}]
Let $\pi^*$ be an optimal transport plan between $\mu$ and $\mu_n$, that is, $\pi^*$ is such that 
\[W_2^2(F_\sharp \mu , F_{n \sharp} \mu_n) = \int_{\M \times \M} d_\M(x,y)^2 d\pi^*(x,y). \]
From Lemma \ref{WassersteinIneq} we know that
\[ W_2^2(F_\sharp \mu, \widetilde F_{n \sharp} \mu_n ) \leq \int_{\M \times \M } |F(x) - F_n(y)|^2 d\pi^*(x,y)= \sum_{l=1}^L\int_{\M \times \M } |u_l(x) - f_l(y)|^2 d\pi^*(x,y). \] 
Using the regularity of the eigenfunctions $f_l$ we can bound the above with
\begin{align*}
\begin{split}
&2\sum_{l=1}^L\int_{\M \times \M } |u_l(x) - f_l(y)|^2 d\pi^*(x,y) +2 C \sum_{l=1}^L\int_{\M \times \M } d_\M(x,y)^2 d\pi^*(x,y) 
\\&\leq   2\sum_{l=1}^L\lVert u_l- f_l  \rVert_{L^2(\mu_n)}^2    +   2CL W_2^2(\mu, \mu_n).
\end{split}
\end{align*}
Combining with the results from Theorem \ref{thm:evectorsveps} we get the desired inequality. 

\end{proof}

%
%
%
%
%
%
%

\appendix

\section{Proof of Proposition \ref{prop:AuxiliaryDensity} }

\label{Appendix}

\begin{proof}
	
First we notice that for small enough but macroscopic (i.e. fixed) $r>0$ we can find a partition of $\M$ (up to overlaps of $\vol_\M$ measure zero) into closed sets $V_1, \dots, V_L$  ($L$ depends on $r$) for which:
\begin{itemize}
\item For every $l=1, \dots, L$, there is a bi-Lipschitz homeomorphism $\Phi_l: V_l \longrightarrow \overline{B(0, r/2)} \subseteq \R^m$ with bi-Lipschitz constant less than $18$. 
\end{itemize}
Such partition can be constructed using a covering of $\M$ with balls of radius $r$. The centers of the balls can then be used to construct a Voronoi tessellation of $\M$, inducing in this way the sets $V_i$. This construction is presented in detail in Propositions 2 and 3 in \cite{trillos2018spectral}. Since the regions $V_i$ are bi-Lipschitz homeomorphic to $B(0, r/2)$, we can also construct a bi-Lipschitz homeomorphism 
\[\Phi_l: V_l \longrightarrow [0,r]^m  \]
with bi-Lipschitz constant less than $C$ (which only depends on $m$).

Let $\widetilde \delta < \frac{r}{4}$ and let us consider a partition of $[0,1]^m$ into a collection of ($m$-dimensional) rectangles $\widetilde{Q}$ with diameter less than $\widetilde \delta$ and with aspect ratio no larger than $2$. We construct subsets $Q$ of $\M$ by letting
\[ Q :=  \Phi_l^{-1}( \widetilde Q )   \]
for some $l=1, \dots, L$ and some $\widetilde Q$. From the properties of the maps $\Phi_l$ it follows that the collection of cells $Q$ is a partition of $\M$ (the different cells are disjoint up to sets of $\vol_\M$ measure zero) and satisfy:
\begin{enumerate}
	\item $\diam(Q) \leq C \widetilde \delta$.
	\item $ \widetilde \delta^m \leq C \vol_\M(Q) \leq C \mu \left(Q \right) $.
\end{enumerate}
We can then let $\widetilde \rho_n : \M \rightarrow $ be the piecewise constant density defined by
\[ \widetilde \rho_n(x) = \frac{\mu_n(Q)}{\vol_\M(Q)}, \quad x \in Q.\]
It is clear that $\widetilde \rho_n$ is indeed a density and moreover, for every $x \in Q$ we have 
\[  | \rho(x) - \widetilde \rho_n(x) | \leq  \left |  \frac{\mu_n(Q) - \mu(Q)}{\vol_\M(Q)}  \right| +  \frac{1}{\vol_\M(Q)} \left | \int_{Q}(\rho(y) - \rho(x)) d \vol_\M(y)  \right|     \]
In particular, 
\[ \lVert \rho - \widetilde \rho_n
 \rVert_{L^\infty(\mu)} \leq  \frac{C}{\widetilde \delta^m} \sup_{Q} | \mu_n(Q) - \mu(Q)|  +C \Lip(\rho) \widetilde \delta .  \]

From Chernoff's bound one can get for all $\theta$ small enough, 
 \begin{align*}
   \Prob \left(  \mu_n(Q)  \geq  (1+\theta)\mu(Q)  \right) & \leq \exp \left(  -  \frac{n\theta^2\mu(Q)}{3}     \right)
   \\ &\leq \exp \left(  -  \frac{Cn\theta^2\widetilde \delta^m}{3}     \right)
 \end{align*}
 where in the last line we have used the fact that
 \[ \rho_{min} \vol_\M(Q) \leq \mu(Q)   \]
 and that
 \[ C\widetilde \delta^m \leq \vol_\M(Q). \]
 Likewise we can obtain
 \begin{align*}
 \Prob \left(  \mu_n(Q)  \leq  (1-\theta)\mu(Q)  \right) & \leq \exp \left(  -  \frac{Cn\theta^2\widetilde \delta^m}{3}     \right),
 \end{align*}
 and so
 \[  \Prob \left(  |\mu_n(Q) - \mu(Q) |  \geq  \theta\mu(Q)  \right) \leq  2\exp \left(  -  \frac{Cn\theta^2\widetilde \delta^m}{3}     \right).   \]
Taking a union bound, we deduce that
\[ \Prob \left(  | \mu(Q) - \mu_n(Q) | \geq \theta \vol_\M(Q) \quad \forall Q \right) \leq  \frac{L}{\theta^m} \exp \left( -C n\theta^2 \widetilde \delta^m          \right). \]
and as long as $\theta\geq \frac{1}{n^{1/m}}$ we can replace $L/\theta^m$ above with $n$.
	
\end{proof}

\bibliographystyle{abbrv}
\bibliography{ref}

\end{document}